\documentclass[12pt, leqno]{article}

\title{Voros Coefficients for the Hypergeometric Differential Equations 
and Eynard-Orantin's Topological Recursion \\
---{\Large Part II : For the Confluent Family of Hypergeometric Equations ---}
}
\author{
Kohei \textsc{Iwaki}$^{\dagger}$\and
Tatsuya \textsc{Koike}$^{\sharp}$\and
Yumiko \textsc{Takei}$^{\heartsuit}$}
\def\paperinfo{
\renewcommand{\thefootnote}{\fnsymbol{footnote}}
\footnote[0]{$^{\dagger}$Graduate School of Mathematics, 
Nagoya University. 
}
\footnote[0]{$^{\sharp}$Department of Mathematics, 
Graduate School of Science, Kobe University.}
\footnote[0]{$^{\heartsuit}$Department of Mathematics, 
Graduate School of Science, Kobe University. 
}
\footnote[0]{2010 \textit{Mathematics Subject Classification}. 
Primary:34M60; Secondary:81T45}
\footnote[0]{\textit{Keywords}: Exact WKB analysis; Voros coefficients;
Topological recursion; Quantum curves; Free energy; Gauss hypergeometric equaiton.}
\renewcommand*{\thefootnote}{\arabic{footnote}}
}
\date{\today}

\usepackage{afterpage}
\usepackage{amsmath}
\usepackage{amsthm}
\usepackage{amssymb}        
\usepackage{amsfonts}       
\usepackage[dvipdfmx]{color}
\usepackage[dvipdfmx]{graphicx}
\usepackage{enumitem}
\usepackage[mathscr]{eucal} 
\usepackage[a4paper]{geometry}
\usepackage{lipsum}
\usepackage{lscape}
\usepackage{mathrsfs}
\usepackage{stmaryrd}       
\usepackage[all]{xy}

\theoremstyle{plain}
\newtheorem{thm}{Theorem}[section]
\newtheorem{prop}[thm]{Proposition}
\newtheorem{lem}[thm]{Lemma}

\theoremstyle{definition}
\newtheorem{dfn}[thm]{Definition}
\theoremstyle{remark}
\newtheorem{rem}[thm]{Remark}

\numberwithin{equation}{section}
\numberwithin{table}{section}
\numberwithin{figure}{section}

\def\Res{\mathop{\rm{Res}}}

\def\ord{\mathop{\rm{ord}}\nolimits}

\def\Sing{\mathop{\rm{Sing}}\nolimits}

\begin{document}

\maketitle

\paperinfo

\begin{abstract}
We show that the each member of the confluent family of 
the Gauss hypergeometric equations is realized as 
quantum curves for appropriate spectral curves.
As an application, relations between the Voros coefficients
of those equations and the free energy of their classical limit
computed by the topological recursion are established. 
We will also find explicit expressions of the free energy and the Voros coefficients 
in terms of the Bernoulli numbers and Bernoulli polynomials. 

\end{abstract}

\tableofcontents


\section{Introduction} 
\label{section:intro}

This is the second part of \cite{IKT-part1} 
on a relation between 
the topological recursion and the exact WKB analysis.
Let us describe our motivation again 
although it was explained in the previous article.

The theory of quantum curves, developed by many authors including 
\cite{GS, DM, BE}, is a bridge between the topological recursion 
(or matrix models) and the WKB analysis. 
It claims that, under a certain admissibility assumption (\cite{BE}), 
a generating function 
\begin{equation} \label{eq:wave-z-intro}
\psi(z,\hbar) = \exp \left( \sum_{\substack{g \ge 0 \\ n \ge 1}} 
\frac{\hbar^{2g-2+n}}{n!} \int^z \cdots \int^z 
\left( W_{g,n}(z_1,\dots, z_n) - \delta_{g,0}\delta_{n,2} 
\frac{dx(z_1)dx(z_2)}{(x(z_1) - x(z_2))^2} \right)
\right)
\end{equation}
of Eynard-Orantin's correlation functions 
$\{ W_{g,n}(z_1,\dots,z_n) \}_{g \ge 0, n \ge 1}$ defined by 
the topological recursion (\cite{EO}) 
gives the WKB solution of a Schr{\"o}dinger-type linear differential 
(or difference) equation whose classical limit 
coincides with the initial spectral curve for the topological recursion. 
It is intriguing because the framework of the quantum curve also suggests
that the expansion coefficients of WKB solution 
may contain an information of various enumerative invariants 
which the topological recursion computes 
(see \cite{Eyanrd-11, EO2, EO-08, Zhou12} etc).

The purpose of this paper is to add a new perspective to the 
theory of quantum curves from the view point of the exact WKB analysis
(cf.\,\cite{Voros83, DDP93, KT98}). 
In particular, as a continuation of the authors' work \cite{IKT-part1}, 
we investigate a relationship between 
the following two objects:
\begin{itemize}
\item
The free energy 
\begin{equation}
F = \sum_{g=0}^{\infty} \hbar^{2g-2}F_g
\end{equation}
of spectral curves. Here $F_g$'s are computed by the topological recursion
(see \S \ref{section:free-energy-part-2}).
\item
The Voros coefficient 
\begin{equation}
V = \int_{\gamma} \left(S(x,\hbar) - \hbar^{-1} S_{-1}(x) - S_0(x) \right) \, dx
= \sum_{m=1}^{\infty} \hbar^{m} \int_\gamma S_m(x) \, dx
\end{equation}
in the theory of the exact WKB analysis. 
Here $S(x,\hbar) = \sum_{m\ge -1} \hbar^{m} S_m(x)$ 
is the logarithmic derivative of the WKB solution, 
and $\gamma$ is a certain integration contour
(see \S \ref{section:def-Voros}). 
\end{itemize}
The free energy 
is one of the most important quantity in 
the topological recursion because it is closely related to
the partition function of matrix models.
On the other hand, after the work of \cite{Voros83} and \cite{DDP93}, Voros coefficients 
have been recognized as one of the most important objects in the exact WKB analysis. 
For example, it is known that the global behavior (monodoromy or connection matrices) of solutions 
of Schr{\"o}dinger-type equations is described by 
Borel resummed Voros coefficients (\cite[\S 3]{KT98}). 
Voros coefficients also play an essential role in the analysis of parametric Stokes phenomenon 
(cf.\,\cite{AKT09, KoT11, AIT, I14} etc.) 
which is closely related to a cluster algebraic structure 
in the exact WKB analysis (\cite{IN14}).

In previous works for quantum curves, the roles of the free energy 
and Voros coefficients are not clear even though they are crucially important 
in the topological recursion or the exact WKB analysis. 
In our first paper \cite{IKT-part1} we started our study to answer the following question:
\begin{quote}
What quantity corresponds to the Voros coefficient in the topological recursion ?
\end{quote}
One of the main result of \cite{IKT-part1} is to establish a relation between 
the free energy and the Voros coefficient for a specific spectral curve
$y^2 = (x^2/4) - \lambda_\infty$, called the Weber curve. 
In this case, the associated quantum curve becomes 
the Weber equation (harmonic oscillator), and we find that the 
Voros coefficient for the Weber equation is computed as the difference value 
of the free energy. Since we know an explicit expression of the Voros coefficient 
for the Weber equation thanks to the work \cite{Takei08}, 
we could also obtain an explicit formula for the free energy. 
In particular, our computation recovered the well-known formula 
for the free energy of the Weber curve in terms of the Bernoulli numbers
given by \cite{HZ, Penner}. 
As our main result of the second part, we will generalize the result 
for the Weber curve to the spectral curves arising from the confluent family 
of the Gauss hypergeometric differential equations with an appropriately introduced 
small parameter $\hbar$. 
Such equations have already been considered from 
the view point of the exact WKB analysis by many authors 
(Gauss:\,\cite{Aoki-Tanda}, 
Kummer:\,\cite{Takahashi, ATT}, 
Legendre:\,\,\cite{Ko3, KKKT11}, 
Whittaker:\,\cite{KoT11, KKKT11}, 
Bessel:\,\cite{AIT},
Weber:\,\cite{Takei08, AKT09} etc.).

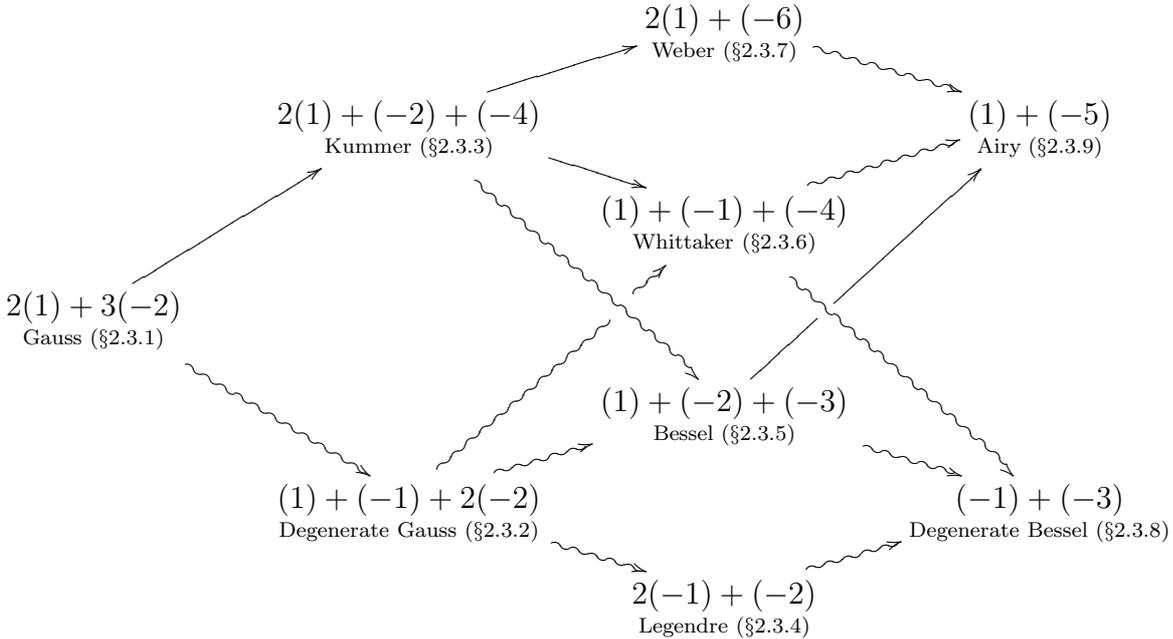
\begin{figure}[t]
$$
\xymatrix@!C=35pt@R=7pt{
&&
&& \underset{\text{Weber (\S\ref{subsection:quantum-Weber})}}{2(1) + (-6)} \ar@{~>}[rrd]
&&
\\
&& \underset{\text{Kummer (\S\ref{subsection:quantum-Kummer})}}
{2(1) + (-2) + (-4)} \ar@{->}[rru] \ar@{->}[rrd] \ar@{~>}[rrddd]
&&
&& \underset{\text{Airy (\S\ref{subsection:quantum-Airy})}}{(1) + (-5)}
\\
&&
&& \underset{\text{Whittaker (\S\ref{subsection:quantum-Whittaker})}}
{(1)+(-1)+(-4)} \ar@{~>}[rru] \ar@{~>}[rrddd]
&&
&&
\\
\underset{\text{Gauss (\S\ref{subsection:quantum-Gauss})}}
{2 (1) + 3(-2)} \ar@{->}[rruu] \ar@{~>}[rrdd]
&&
&&
&&
&&
\\
&&
&& \underset{\text{Bessel (\S\ref{subsection:quantum-Bessel})}}{(1)+(-2)+(-3)} 
\ar@{->}[rruuu]|(0.32)\hole \ar@{~>}[rrd]
&&
&&
\\
&& \underset{\text{Degenerate Gauss (\S\ref{subsection:quantum-d-Gauss})}}
{(1)+(-1)+2(-2)} 
\ar@{~>}[rru] \ar@{~>}[rrd] \ar@{~>}[rruuu]|(0.67)\hole
&&
&& \underset{\text{Degenerate Bessel (\S\ref{subsection:quantum-d-Bessel})}}{(-1) + (-3)}
\\
&&
&& \underset{\text{Legendre (\S\ref{subsection:quantum-Legendre})}}{2(-1) + (-2)}\ar@{~>}[rru]
&&
\\
}
$$
\caption{The confluence diagram of spectral curves 
obtained as classical limit of the confluent family of the
Gauss hypergeometric equations. 
A straight line (resp., a wiggly line) in the figure
denotes the confluence of singular points
(resp., the coalescence of a turning point and a singular point).
Here a simple pole of the potential is also regarded 
as a turning point (cf. \cite{Ko1, Ko2}). 
Although the degenerate Bessel equation
in \S \ref{subsection:quantum-d-Bessel} is a special 
(i.e. $\lambda_0=0$) case of the Bessel equation in 
\S \ref{subsection:quantum-Bessel}, 
we distinguish these two because the geometry of 
the spectral curves are different.}
\label{fig:confluence-WKB}
\end{figure}

Before stating our results, let us show the confluence diagram, 
shown in Fig.\,\ref{fig:confluence-WKB}, 
of spectral curves considered in this paper. 
The spectral curves are classical limit of the
confluent family of the Gauss hypergeometric equations 
with a small parameter $\hbar$.
It takes the form $y^2 - Q_{\rm cl}(x) = 0$
and concrete expressions of the rational function $Q_{\text{cl}}(x)$ 
are given in Table \ref{table:classical}. 
The numbers in the diagram in Fig.\,\ref{fig:confluence-WKB} 
are related to orders of zeros and poles of 
the meromorphic quadratic differential $Q_{\text{cl}}(x)(dx)^2$. 
More precisely, in the diagram in Fig.\,\ref{fig:confluence-WKB} 
the number in a parenthesis denotes an order of a pole
of $Q_{\text{cl}}(x)(dx)^2$,
and the number before parenthesis denotes 
a number of poles of this order. 
For example, a symbol $2 (1) + (-2) + (-4)$ means that 
$Q_{\text{cl}}(x)(dx)^2$ has two simple zeros, 
one pole of order two, and one pole of order four. 
Zeros of $Q_{\text{cl}}(x)(dx)^2$ are nothing but turning points in the WKB analysis.

\begin{table}[t]
\begin{center}
\begin{tabular}{ccc}\hline
 & $Q_{\text{cl}}(x)$  & 
\\\hline\hline
\parbox[c][4.5em][c]{0em}{}
Gauss (\S\ref{subsection:quantum-Gauss})
&
\begin{minipage}{.35\textwidth}
\begin{center}
$\dfrac{{\lambda_\infty}^2 x^2 
- ({\lambda_\infty}^2 + {\lambda_0}^2 - {\lambda_1}^2)x 
+ {\lambda_0}^2}{x^2 (x-1)^2}$
\end{center}
\end{minipage}
&
\begin{minipage}{.35\textwidth}
\begin{center}
$\lambda_0, \lambda_1, \lambda_\infty \neq 0$,\\
$\lambda_\infty \neq \lambda_0 \pm \lambda_1$,\\
$\lambda_\infty \neq - (\lambda_0 \pm \lambda_1)$.
\end{center}
\end{minipage}
\\\hline
\parbox[c][3.5em][c]{0em}{}
Degenerate Gauss
(\S\ref{subsection:quantum-d-Gauss})
&
\begin{minipage}{.35\textwidth}
\begin{center}
$\dfrac{{\lambda_\infty}^2 x + {\lambda_1}^2 - {\lambda_\infty}^2}{x(x-1)^2}$
\end{center}
\end{minipage}
&
\begin{minipage}{.35\textwidth}
\begin{center}
$\lambda_1, \lambda_\infty \neq 0$,\\
$\lambda_\infty \neq \pm \lambda_1$.
\end{center}
\end{minipage}
\\\hline
\parbox[c][3.5em][c]{0em}{}
Kummer (\S\ref{subsection:quantum-Kummer})
&
\begin{minipage}{.35\textwidth}
\begin{center}
$\dfrac{x^2 + 4 \lambda_\infty x + 4 {\lambda_0}^2}{4x^2}$
\end{center}
\end{minipage}
&
\begin{minipage}{.35\textwidth}
\begin{center}
$\lambda_0\neq 0$,\\
$\lambda_\infty \neq \pm \lambda_0$.
\end{center}
\end{minipage}
\\\hline
\parbox[c][3.5em][c]{0em}{}
Legendre (\S\ref{subsection:quantum-Legendre})
& $\dfrac{\lambda_\infty^2}{x^2-1}$
&
$\lambda_\infty \neq 0$.
\\\hline
\parbox[c][3.5em][c]{0em}{}
Bessel (\S\ref{subsection:quantum-Bessel})
& $\dfrac{x + \lambda_0^2}{4x^2}$
&
$\lambda_0 \neq 0$.
\\\hline
\parbox[c][3.5em][c]{0em}{}
Whittaker (\S\ref{subsection:quantum-Whittaker})
& $\dfrac{x - 4\lambda_\infty}{4 x}$
& $\lambda_\infty \neq 0$.
\\\hline
\parbox[c][3.5em][c]{0em}{}
Weber (\S\ref{subsection:quantum-Weber})
& $\dfrac{1}{4} x^2 - \lambda_\infty$
& $\lambda_\infty \neq 0$.
\\\hline
\parbox[c][3.5em][c]{0em}{}
Degenerate Bessel (\S \ref{subsection:quantum-d-Bessel})
& $\dfrac{1}{x}$
& 
\\\hline
\parbox[c][3.5em][c]{0em}{}
Airy (\S \ref{subsection:quantum-Airy})
& $x$
& 
\\\hline
\end{tabular}
\end{center}
\caption{Classical limit 
$y^2 - Q_{\text{cl}}(x) = 0$ 
of the quantum curves 
discussed in \S \ref{section:example-quantum-curves}.}
\label{table:classical}
\end{table}

Note that $Q_{\rm cl}(x)$ contains a tuple of parameters $\lambda_j$ 
which determines the formal monodoromy (or the characteristic exponents) 
around singular points. We impose the conditions shown in Table \ref{table:classical} 
on these parameters so that the topological recursion is applicable when we
regard these curves as spectral curves.
Then, the other result in \cite{IKT-part1} 
(cf. \cite[Theorem 3.5]{IKT-part1}, or Theorem \ref{thm:WKB-Wg,n}) shows that, 
after a coordinate change $z \mapsto x(z)$, 
the formula \eqref{eq:wave-z-intro} gives a WKB solution of 
a second order linear ODE (i.e., the quantum curve associated with 
the spectral curve $y^2 - Q_{\rm cl}(x) = 0$) 
which is equivalent to a corresponding member 
in the confluent family of the Gauss hypergeometric equations. 
Therefore, the diagram in Fig.\,\ref{fig:confluence-WKB} can be 
identified with the traditional confluence diagram for hypergeometric equations.
Concrete forms of the quantum curves are given in \S \ref{section:example-quantum-curves}.
Note that, although the Gauss curve in Table \ref{table:classical} 
is not admissible in the sense of \cite{BE}, the previous result
\cite[Theorem 3.5]{IKT-part1} is applicable to this example 
and we may check that the resulting quantum curve is equivalent to
the Gauss hypergeometric equation. See \S \ref{subsection:quantum-Gauss}.

Let us formulate our result for the Gauss curve 
\begin{equation} \label{eq:Gauss-curve-intro}
y^2 - \frac{ {\lambda_{\infty}}^2 x^2 - 
({\lambda_{\infty}}^2 + {\lambda_0}^2 - {\lambda_1}^2)x + {\lambda_0}^2 }{x^2 (1 - x)^2} = 0
\end{equation} 
and the associated quantum curve (quantum Gauss curve) given in \eqref{eq:Gauss_eq(d/dx)}. 
We also set the parameters $\nu_{j\pm} = 1/6$ ($j \in \{0,1,\infty \}$) 
in the quantum Gauss curve \eqref{eq:Gauss_eq(d/dx)} for simplicity.
Since the quantum Gauss curve has three singular points at $x=0,1,\infty$ on ${\mathbb P}^1$, 
three Voros coefficients $V^{(j)}(\underline{\lambda}; \hbar)$
($\underline{\lambda} = (\lambda_0,\lambda_1,\lambda_\infty)$) 
are associated depending on the choice of singular point 
$j \in \{0,1, \infty\}$ (see \S \ref{subsec:main-theorem} for precise definition). 
Having this in mind, we find the following relation between 
the Voros coefficient and the free energy for the Gauss curve: 

\begin{thm}[Theorem \ref{thm:main-theorem-in-part2} (i)]
The Voros coefficient $V^{(j)}(\underline{\lambda}; \hbar)$ 
for the singular point $j \in \{0,1,\infty \}$ of the quantum Gauss curve \eqref{eq:Gauss_eq(d/dx)} 
and the free energy $F(\underline{\lambda})$ of the Gauss curve 
\eqref{eq:Gauss-curve-intro} are related as follows:
\begin{align}
& V^{(0)}(\underline{\lambda}; \hbar)
= F\left(\lambda_0+\frac{\hbar}{2}, \lambda_1, \lambda_{\infty}; \hbar\right)
- F\left(\lambda_0-\frac{\hbar}{2}, \lambda_1, \lambda_{\infty}; \hbar\right)
-\frac{1}{\hbar}\frac{\partial F_0}{\partial\lambda_0},
\\
& V^{(1)}(\underline{\lambda}; \hbar)
= F\left(\lambda_0, \lambda_1+\frac{\hbar}{2}, \lambda_{\infty}; \hbar\right)
- F\left(\lambda_0, \lambda_1-\frac{\hbar}{2}, \lambda_{\infty}; \hbar\right)
- \frac{1}{\hbar}\frac{\partial F_0}{\partial\lambda_1},
\\
& V^{(\infty)}(\underline{\lambda}; \hbar)
= F\left(\lambda_0, \lambda_1, \lambda_{\infty}+\frac{\hbar}{2}; \hbar\right)
- F\left(\lambda_0, \lambda_1, \lambda_{\infty}-\frac{\hbar}{2}; \hbar\right)
- \frac{1}{\hbar}\frac{\partial F_0}{\partial\lambda_\infty}.
\end{align}
\end{thm}

To obtain the above results, the expression 
\eqref{eq:wave-z-intro} plays an essential role.
Note that a similar result is valid for all examples in Table \ref{table:classical}.
Although there is a confluence diagram in Fig.\,\ref{fig:confluence-WKB}, 
we employed case-by-case study; this is because a coalescence of ramification point 
and poles of $W_{0,1}(z)$ happens in the confluent procedure 
for which we need more careful treatment.

It should be noted that, although several Voros coefficients are associated
with a given differential equation (having several singular points) 
as in the case of Gauss hypergeometric differential equation, 
they are described in terms of the generating function of the free energies 
which is canonically (or uniquely) associated with a given spectral curve.
These results imply that the free energy is more fundamental object that 
controls the Voros coefficient. 

The above formula has several applications.
Firstly, by combining the above relation and 
the contiguity relations for hypergeometric equations 
(cf. \S \ref{sec:contiguity-relation}),
we can derive three-term difference equations 
which the free energy satisfies (Theorem \ref{thm:main-theorem-in-part2} (ii)). 
By solving them, we obtain concrete forms of the free energies and
those of the Voros coefficients as well 
(Theorem \ref{thm:main-theorem-in-part2} (iii) and (iv)). 
The explicit form of the $g$-th free energy $F_g$ 
is shown in Table \ref{table:free-energy}.  
Interestingly, the all free energies discussed in this paper 
are expressed in terms of the Bernoulli numbers. 
Consequently, the Voros coefficients of the quantum curves 
are written in terms of the Bernoulli polynomials. 
We note that the idea to use the contiguity relation 
to obtain the explicit formula for Voros coefficients 
are due to \cite{Takei08}.
We hope the technique to compute the free energy and the Voros coefficient 
shown in this paper is applicable to more wide class of 
spectral curves and the associated quantum curves.

\begin{table}[t]
\begin{center}
\begin{tabular}{cc}\hline
 & $F_g$ $(g \geq 2)$
\\\hline\hline
\parbox[c][3.5em][c]{0em}{}
Gauss (\S\ref{subsection:quantum-Gauss})
&
$\dfrac{B_{2g}}{2g (2g-2)}
\left\{
\dfrac{1}{(\lambda_0 + \lambda_1 + \lambda_{\infty})^{2g-2}}
+ \dfrac{1}{(\lambda_0 - \lambda_1 + \lambda_{\infty})^{2g-2}}\right.$
\\[-.5em]
\parbox[c][3.5em][c]{0em}{}
&
\hspace{7em}
$\left.
+ \dfrac{1}{(\lambda_0 + \lambda_1 - \lambda_{\infty})^{2g-2}}
+ \dfrac{1}{(\lambda_0 - \lambda_1 - \lambda_{\infty})^{2g-2}}
\right.$
\\[-.5em]
\parbox[c][3.5em][c]{0em}{}
&
\hspace{5em}
$\left.
- \dfrac{1}{(2 \lambda_0)^{2g-2}}
- \dfrac{1}{(2 \lambda_1)^{2g-2}}
- \dfrac{1}{(2 \lambda_{\infty})^{2g-2}}
\right\}$
\\\hline
\parbox[c][3.5em][c]{0em}{}
\begin{minipage}{.15\textwidth}
Degenerate Gauss
(\S\ref{subsection:quantum-d-Gauss})
\end{minipage}
&
$\dfrac{B_{2g}}{2g (2g-2)}
\left\{
\dfrac{2}{(\lambda_1 + \lambda_\infty)^{2g-2}}
+ \dfrac{2}{(\lambda_1 - \lambda_\infty)^{2g-2}}
- \dfrac{1}{(2 \lambda_1)^{2g-2}}
- \dfrac{1}{(2 \lambda_\infty)^{2g-2}}
\right\}$
\\\hline
\parbox[c][3.5em][c]{0em}{}
Kummer 
(\S\ref{subsection:quantum-Kummer})
&
$\dfrac{B_{2g}}{2g (2g-2)}
\left\{
\dfrac{1}{(\lambda_0 + \lambda_\infty)^{2g-2}}\right.
\left.+ \dfrac{1}{(\lambda_0 - \lambda_\infty)^{2g-2}} -
 \dfrac{1}{(2\lambda_0)^{2g-2}}\right\}$
\\\hline
\parbox[c][3.5em][c]{0em}{}
Legendre 
(\S\ref{subsection:quantum-Legendre})
& $\dfrac{B_{2g}}{2g(2g-2)} \left\{ \dfrac{4}{\lambda_\infty^{2g-2}} 
 - \dfrac{1}{(2 \lambda_\infty)^{2g-2}} \right\}$
\\\hline
\parbox[c][3.5em][c]{0em}{}
Bessel (\S\ref{subsection:quantum-Bessel})
& $- \dfrac{B_{2g}}{2g (2g-2)} \dfrac{1}{(2\lambda_0)^{2g-2}}$
\\\hline
\parbox[c][3.5em][c]{0em}{}
Whittaker 
(\S\ref{subsection:quantum-Whittaker})
& $\dfrac{B_{2g}}{2g (2g-2)} \dfrac{2}{\lambda_\infty^{2g-2}}$
\\\hline
\parbox[c][3.5em][c]{0em}{}
Weber 
(\S\ref{subsection:quantum-Weber})
& $\dfrac{B_{2g}}{2g (2g-2)} \dfrac{1}{\lambda_\infty^{2g-2}}$
\\\hline
\end{tabular}
\end{center}
\caption{Free energies for the Spectral curves $y^2 - Q_{\text{cl}}(x, \underline{\lambda}) = 0$
in Table \ref{table:classical}.
In this table $B_{2g}$ denotes the $2g$-th Bernoulli number 
(see \eqref{def:Bernoulli} for the definition).}
\label{table:free-energy}
\end{table}

The paper is organized as follows: 
In \S\ref{sec:from-part-I} we recall several notions in the 
topological recursion, and the quantum curve theorem (cf. \cite[Theorem 3.5]{IKT-part1}). 
The quantum curves associated with the spectral curves in Table \ref{table:classical} 
are also computed. 
In \S\ref{sec:Voros-vs-TR}, after introducing the notion of the Voros coefficient, 
we state our main theorem (Theorem \ref{thm:main-theorem-in-part2}). 
We give a proof of our result only for Gauss curve, 
but the other examples can be treated similarly.
For the convenience of readers, in appendix 
we summarize several formulas on contiguity relations (\S \ref{sec:contiguity-relation}) and 
on Bernoulli numbers/polynomials (\S \ref{sec:bernoulli}) 
which will be used to prove our main results.

\section*{Acknowledgement}
We are grateful to 
Takashi Aoki, 
Takahiro Kawai,
Toshinori Takahashi,
Yoshitsugu Takei 
and 
Mika Tanda
for helpful discussions and communications.
This work is supported, in part, by JSPS KAKENHI Grand Numbers 
16K17613, 16H06337, 16K05177, 17H06127.

\section{Topological recursion and quantum curves : Results from Part I}
\label{sec:from-part-I}

In this section we briefly recall a result obtained 
in the Part I \cite{IKT-part1}. Definitions of some notions which 
is relevant in Part II are also recalled. 

\subsection{Preliminaries of the topological recursion}
\label{sec:TR}

\subsubsection{Definition of correlation functions}
Let us briefly recall several facts on 
Eynard-Orantin's theory (\cite{EO}) 
together with assumptions imposed on \cite{IKT-part1}.

Let us start from the definition of genus $0$ spectral curves. 
(We adopt restricted situation; see \cite{EO} for general definition 
of spectral curves.)
\begin{dfn}
A spectral curve (of genus $0$) is a pair $(x(z), y(z))$
of non-constant rational functions on $\mathbb{P}^1$, 
such that their exterior differentials 
$dx$ and $dy$ never vanish simultaneously. 
\end{dfn}

Let $R$ be the set of ramification points of $x(z)$,
i.e., $R$ consists of zeros of $dx(z)$ of 
any order and poles of $x(z)$
whose orders are greater than or equal to two
(here we consider $x$ as a branched covering map
from $\mathbb{P}^1$ to itself).
We also assume that all ramification points of $x(z)$ 
are simple so that the local conjugate map $z \mapsto \bar{z}$ 
near each ramification point is well-defined.

In our situation, the topological recursion is 
formulated as follows. 
(See \cite{EO} for general situation; 
e.g., when the spectral curve has higher genus.)

\begin{dfn}[{\cite[Definition 4.2]{EO}}]
Eynard-Orantin's correlation function
$W_{g, n}(z_1, \cdots, z_n)$ for $g \geq 0$ and $n \geq 1$
is defined
as a multidifferential on $(\mathbb{P}^1)^n$ 
using the recurrence relation
(called Eynard-Orantin's topological recursion)
\begin{align}
\label{eq:TR}
W_{g, n+1}(z_0, z_1, \cdots, z_n)
&:= \sum_{r \in R}
\Res_{z = r} K_r(z_0, z)
\Bigg[
W_{g-1, n+2} (z, \overline{z}, z_1, \cdots, z_n)
\\
&\qquad\qquad
+
\sum'_{\substack{g_1 + g_2 = g \\ I_1 \sqcup I_2 = \{1, 2, \cdots, n\}}}
W_{g_1, |I_1| + 1} (z, z_{I_1})
W_{g_2, |I_2| + 1} (\overline{z}, z_{I_2})
\Bigg]
\notag
\end{align}
for $2g + n \geq 2$ with initial conditions
\begin{align}
W_{0, 1}(z_0) &:= y(z_0) dx(z_0),
\quad
W_{0, 2}(z_0, z_1) = B(z_0, z_1)
:= \frac{dz_0 dz_1}{(z_0 - z_1)^2}.
\end{align}
Here we set $W_{g,n} \equiv 0$ for a negative $g$,
\begin{equation}
\label{eq:RecursionKernel}
K_r(z_0, z)
:= \frac{1}{2\big(y(z) - y(\overline{z})\big) dx(z)}
\int^{\zeta = z}_{\zeta = \overline{z}} B(z_0, \zeta)
\end{equation}
is a recursion kernel defined near a ramification point $r \in R$,
$\sqcup$ denotes the disjoint union,
and
the prime ${}'$ on the summation symbol in \eqref{eq:TR}
means that we exclude terms for
$(g_1, I_1) = (0, \emptyset)$
and
$(g_2, I_2) = (0, \emptyset)$
(so that $W_{0, 1}$ does not appear) in the sum.
We have also used the multi-index notation:
for $I = \{i_1, \cdots, i_m\} \subset \{1, 2, \cdots, n\}$
with $i_1 < i_2 < \cdots < i_m$, $z_I:= (z_{i_1}, \cdots, z_{i_m})$.
\end{dfn}

See \cite{EO} for properties of $W_{g,n}$. 
We define ineffective ramification points 
as ramification points 
at where $W_{g,n}$ with $2g-2+n\ge 0$
is holomorphic in each variable. 
We also introduce the set of effective ramification points
$R^{\ast} \subset R$ consisting of ramification points 
which are not ineffective (cf. \cite[Definition 2.6]{IKT-part1}). 
Ineffective ramification points often appear as 
poles of $x(z)$, and those points can be chosen as 
an end-point of the integration path for $W_{g,n}$ 
when we discuss quantum curves 
(see \S \ref{subsection:quantum-curve} below).

\subsubsection{Definition of free energies}
\label{section:free-energy-part-2}

The free energy $F_g$ ($g\geq 0$) is defined for the spectral curve,
and one of the most important objects in Eynard-Orantin's theory.

\begin{dfn}[{\cite[Definition 4.3]{EO}}]
For $g \geq 2$, the $g$-th free energy $F_g$ is defined by
\begin{equation}
\label{def:Fg2}
F_g := \frac{1}{2- 2g} \sum_{r \in R} \Res_{z = r}
\big[\Phi(z) W_{g, 1}(z) \big],
\end{equation}
where $\Phi(z)$ is a primitive of $y(z) dx(z)$. 
The free energies $F_0$ and $F_1$ for $g=0$ and $1$ 
are also defined, but in a different manner 
(see \cite[\S 4.2.2 and \S 4.2.3]{EO} for the definition). 
\end{dfn}
Note that the right-hand side of \eqref{def:Fg2} does not
depend on the choice of the primitive.
The generating series
\begin{equation} \label{eq:total-free-energy}
F := \sum_{g = 0}^{\infty} \hbar^{2g-2} F_g
\end{equation}
of $F_g$'s is also called
the free energy of the spectral curve.

\subsection{Quantization of spectral curves}
\label{subsection:quantum-curve}

As well as \cite{IKT-part1}, we assume that the spectral curve 
$(x(z), y(z))$ satisfies
\begin{itemize}
\item[(A1)]
A function field $\mathbb{C}(x(z), y(z))$ 
coincides with $\mathbb{C}(z)$.

\item[(A2)]

If $r$ is a ramification point which is a pole of $x(z)$, 
and if $Y(z) = - x(z)^2 y(z)$ is holomorphic near $r$,
then $dY(r) \neq 0$.

\item[(A3)]
All of the ramification points of $x(z)$ are simple,
i.e., the ramification index of each ramification point
is two.

\item[(A4)]
We assume branch points are all distinct,
where a branch point is defined as the image of
a ramification point by $x(z)$.
\end{itemize}

Because of the assumption (A1), we can find an irreducible
polynomial $P(x, y) \in \mathbb{C}[x, y]$ 
for which $P(x(z), y(z)) = 0$
holds at all $z$  except for poles of $x(z)$ and $y(z)$.
We also call this curve $\mathcal{C}$ 
a spectral curve if there is no fear of confusions.

To obtain quantum curves, we further impose 

\begin{description}
\item[{\rm(AQ1)}]
The rational functions $(x(z), y(z))$ satisfy
$P(x(z), y(z)) = 0$ with an irreducible polynomial
$P(x, y) = p_0(x) y^2 + p_1(x) y + p_2(x) \in \mathbb{C}[x, y]$,
where $p_0(x)$ is a nonzero polynomial.

\item[{\rm{(AQ2)}}]
The differential $(y(z) - y(\bar{z})) dx(z)$ does not vanish 
except for at ramification points. 
\end{description}

Then, the map $x: \mathbb{P}^1 \rightarrow \mathbb{P}^1$ 
is a 2-sheeted branched covering, 
and the conjugate map becomes a globally defined 
rational map from $\mathbb{P}^1$ to itself.

Let $W_{g, n}(z_1, \cdots, z_{n})$ be the correlation functions 
of a spectral curve $(x(z),y(z))$ defined 
by the topological recursion. 
Following the idea given in \cite{BE}, 
we associate a formal series defined by
\begin{align}
\label{eq:wave-z}
\varphi(z;\underline{\nu}, \hbar)
&= \exp \Biggl[ \frac{1}{\hbar}\int_{\zeta \in D(z;\underline{\nu})} W_{0, 1}(\zeta)
	+ \frac{1}{2!} \int_{\zeta_1\in D(z ; \underline{\nu})} 
	\int_{\zeta_2 \in D(z ; \underline{\nu})}
   \left( W_{0, 2}(\zeta_1, \zeta_2) 
   - \frac{dx(\zeta_1) \, dx(\zeta_2)}
   {(x(\zeta_1) - x(\zeta_2))^2} \right) \\
& \qquad\qquad
 + \sum_{m = 1}^{\infty} \hbar^m
	\Biggl\{ \sum_{\substack{2g - 2 + n = m \\ g \geq 0, \, n \geq 1}}
		\frac{1}{n!} \int_{\zeta_1 \in D(z ; \underline{\nu})} 
	\cdots \int_{\zeta_n \in D(z ; \underline{\nu})} 
	W_{g, n}(\zeta_1, \cdots, \zeta_n)
	\Biggr\} \Biggr],
\notag
\end{align}
where
\begin{equation}
D(z;\underline{\nu}) = [z] - \sum_{\beta \in B} \nu_\beta [\beta]
\end{equation}
is a divisor on $\mathbb{P}^1$ 
with a tuple $\underline{\nu} = (\nu_{\beta})_{\beta \in B}$ 
of complex numbers satisfying $\sum_{\beta \in B}\nu_{\beta} = 1$, 
and $B \subset \mathbb{P}^1 \setminus R^\ast$ is a finite set. 
For a differential $\omega(z)$, we define its integration 
along the divisor $D(z; \underline{\nu})$ by  
\begin{equation}
\int_{\zeta \in D(z; \nu)} \omega(\zeta) := 
\sum_{k = 1}^m \nu_k \int^{\zeta=z}_{\zeta=\beta} \omega(\zeta),
\end{equation}
and extend the definition to multidifferentials in an obvious way.
Precisely speaking, the above integrals of $W_{0,1}$ and $W_{0,2}$ are not 
well-defined due to the singularity at $\beta \in B$. 
Here we define these integrals through a certain regularization technique; 
e.g., method used in Remark \ref{remark:zeta-reg}.
Since the residue of $W_{g,n}$ at each ramification point
(with respect to each variable $z_1, \cdots, z_n$) vanishes, 
the integrals in \eqref{eq:wave-z} do not depend on 
the choice of paths from $\beta$ to $z$.

In order to give the result of \cite{IKT-part1}
on quantum curve, we introduce
\begin{equation}
\Sing (P) :=\big\{b \in \mathbb{P}^1 \mid \rho (b; P) \leq -2\}, \quad
\Sing_2 (P) :=\big\{b \in \mathbb{P}^1 \mid \rho (b; P) = -2\}
\end{equation}
for a polynomial 
$P(x, y) = p_0(x) y^2 + p_1(x) y + p_2(x) \in {\mathbb C}[x, y]$,
where the index $\rho(b;P)$ of $P$ at $b \in \mathbb{P}^1$ 
is given by
\begin{equation}
\rho(b; P) :=
\begin{cases}
 \ord_{b} Q_0(x) & (b \neq \infty),\\
 \ord_{0} Q_{0}^{(\infty)}(x) & (b = \infty)
\end{cases}
\end{equation}
with
\begin{equation} \label{eq:Q0-in-section-3}
Q_0(x) := \frac{1}{4}\left(
\frac{p_1(x)}{p_0(x)}\right)^2 -  \frac{p_2(x)}{p_0(x)}
\quad\text{and}\quad
Q_0^{(\infty)}(x) := \frac{1}{x^4} Q_0({1}/{x}).
\end{equation}
We also use the following symbols:
\begin{align}
\Delta(z) &:= y(z) - y(\overline{z}), \\
C_{\beta} &:= \Res_{z = \beta} \Delta(z) dx(z).
\end{align}

Then, one of the main result of \cite{IKT-part1} 
is given as follows.
\begin{thm}[{\cite[Theorem 3.5]{IKT-part1}}]
\label{thm:WKB-Wg,n}
We assume (A1) -- (A4), (AQ1) and (AQ2).
Let
\begin{equation} \label{eq:divisor-theorem}
D(z; \underline{\nu}) := [z] - \sum_{\beta \in B}\nu_{\beta} [\beta]
\end{equation}
be a divisor on $\mathbb{P}^1$, where
$B := x^{-1}(\Sing (P))$,
and $\underline{\nu} =(\nu_{\beta})_{\beta \in B}$ is 
a tuple of complex numbers satisfying
$\sum_{\beta \in B} \nu_{\beta} = 1$.
Then $\psi(x;\underline{\nu}, \hbar) := \varphi (z(x); \underline{\nu}, \hbar)$, where
$z(x)$ denotes an inverse function of $x(z)$ and
$\varphi(z, \hbar)$ is defined by \eqref{eq:wave-z}
with the integration divisor \eqref{eq:divisor-theorem},
is a WKB type formal solution of 
\begin{equation}
\label{eq:quantization}
\hat{P} \psi :=
\left( \hbar^2\frac{d^2}{dx^2} + q(x, \hbar) \hbar \frac{d}{dx} 
+ r(x, \hbar) \right) \psi = 0.
\end{equation}
Here
\begin{equation} \label{eq:leading-coeff-of-quantum-curve}
q(x, \hbar) = q_0(x) + \hbar q_1(x),
\quad
r(x, \hbar) = r_0(x) + \hbar r_1(x) + \hbar^2 r_2(x)
\end{equation}
whose leading terms are respectively given by
\begin{equation}
q_0(x) = \frac{p_1(x)}{p_0(x)},
\quad
r_0(x) = \frac{p_2(x)}{p_0(x)}
\end{equation}
and their lower order terms are determined by
\begin{align}
\label{eq:WKB-Wg,n:q1}
x'(z) q_1(x(z))
&=
- \frac{\Delta'(z)}{\Delta(z)}
+ \frac{2}{z - \overline{z}}
- \sum_{\beta \in B}
\frac{\nu_{\beta} + \nu_{\overline{\beta}}}{z - \beta},\\
\label{eq:WKB-Wg,n:r1}
x'(z) r_1(x(z))
&= \frac{1}{2} x'(z) \frac{\partial q_0}{\partial x}\Big|_{x = x(z)}
+ \frac{1}{2} x'(z) q_0(x(z)) q_1(x(z))
+ \frac{1}{2} \Delta(z)
\sum_{\beta \in B}
\frac{\nu_{\beta} - \nu_{\overline{\beta}}}{z - \beta},
\\
\label{eq:WKB-Wg,n:r2}
x'(z) r_2(x(z))
&= \Delta(z) \sum_{\beta \in B_1}
 \frac{\nu_{\beta}\nu_{\overline{\beta}}}{C_{\beta}} 
 \frac{1}{z - \beta}.
\end{align}
Here we set $B_1 := x^{-1}(\Sing_2(P))$.
\end{thm}

It is easy to see that $C_{\beta} \ne 0$ 
for $\beta \in B_1$, and hence, 
the right hand-side of \eqref{eq:WKB-Wg,n:r2} 
is well-defined. 
We call the equation $\hat{P}\psi = 0$ given by 
Theorem \ref{thm:WKB-Wg,n} 
a quantum curve of the spectral curve 
since the classical limit of the equation
$\hat{P}\psi = 0$ coincides with $P(x,y)=0$
which is the initial spectral curve for 
the topological recursion.

\subsection{Confluent family of quantum curves} 
\label{section:example-quantum-curves}

Here we show the list of quantum curves associated with 
the spectral curves arising from the confluent family 
of the hypergeometric equations (see Table \ref{table:classical}). 
We can observe that these equations satisfied by 
classical special functions are recovered as the quantum curve.
An application of the result will be presented in next section.

\subsubsection{Quantum Gauss curve}
\label{subsection:quantum-Gauss}

Let us consider the Gauss curve defined by 
\begin{equation}
\label{eq:Gauss_P(x,y)}
P(x, y)
= x^2 (1 - x)^2 y^2 - \{ {\lambda_{\infty}}^2 x^2 
- ({\lambda_{\infty}}^2 + {\lambda_0}^2 - {\lambda_1}^2)x 
+ {\lambda_0}^2 \} = 0.
\end{equation}
Note that the curve does not satisfy the admissibility condition 
of \cite[Definition 2.7]{BE}; that is, the Newton polygon of 
\eqref{eq:Gauss_P(x,y)} contains an interior point. 
Our spectral curve is obtained as  
a rational parameterization of \eqref{eq:Gauss_P(x,y)} given by
\begin{equation}
\label{eq:Gauss_parameterization}
\begin{cases}
\displaystyle
x = x(z)
= \frac{\sqrt{\Lambda}}{4 {\lambda_{\infty}}^2} 
\left( z + \frac{1}{z} \right)
+ \frac{{\lambda_{\infty}}^2 + {\lambda_0}^2 
- {\lambda_1}^2}{2 {\lambda_{\infty}}^2}
= \frac{\sqrt{\Lambda}}{4 {\lambda_{\infty}}^2} 
\frac{(z - \beta_{0+})(z - \beta_{0-})}{z}, \\[10pt]
\displaystyle
y = y(z)
= \frac{4 {\lambda_{\infty}}^3 z(z^2 - 1)}{\sqrt{\Lambda} 
(z - \beta_{0+})(z - \beta_{0-})(z - \beta_{1+})(z - \beta_{1-})},
\end{cases}
\end{equation}
where
\begin{align} \label{eq:Lambda-Gauss}
\Lambda = 
(\lambda_0 + \lambda_1 + \lambda_{\infty})
(\lambda_0 + \lambda_1 -\lambda_{\infty})
(\lambda_0 - \lambda_1 + \lambda_{\infty})
(\lambda_0 - \lambda_1 - \lambda_{\infty}), \\
\beta_{0\pm} =
- \frac{(\lambda_{0} \pm \lambda_\infty)^2 
- {\lambda_1}^2}{\sqrt{\Lambda}},
\quad
\beta_{1\pm} = \frac{(\lambda_{1} \pm \lambda_\infty)^2 
- {\lambda_0}^2}{\sqrt{\Lambda}}. 
\end{align}
We impose the following generic conditions for parameters 
$\underline{\lambda} = (\lambda_0, \lambda_1, \lambda_\infty)$
so that our assumptions are satisfied:
\begin{equation}
\lambda_0, \lambda_1, \lambda_{\infty} \ne 0 
\quad \text{and} \quad \Lambda \ne 0.
\end{equation}
Note that $R = \{1,-1 \} ~(=R^{\ast})$ and 
the conjugate map is given by $\overline{z} = 1/z$.

We also introduce the notation 
$\beta_{\infty+}:=0$ and $\beta_{\infty-}:=\infty$.
Then, for each $j \in \{0,1,\infty \}$,
we can verify that $\beta_{j \pm}$ are two preimages of 
$j$ by the map $x(z)$; that is, $x(\beta_{j\pm}) = j$ holds. 
Note also that the parameters are chosen so that 
\begin{equation} \label{eq:temperatures}
\pm \lambda_j = \Res_{z = \beta_{j \pm}} y(z) dx(z)
\end{equation}
holds for each $j \in \{0,1,\infty \}$. 
(We will use similar notations in the subsequent subsections so that 
the relation \eqref{eq:temperatures} holds.)

First few terms of the correlation functions and 
free energies are computed as
\begin{align*}
W_{0, 3}(z_1, z_2, z_3)
	&= \biggl\{ \frac{(\beta_{0+} + \beta_{0-} + 2)(\beta_{1+} 
	+ \beta_{1-} + 2)}
	{(z_1 + 1)^2 (z_2 + 1)^2 (z_3 + 1)^2} \\ &\qquad
	- \frac{(\beta_{0+} + \beta_{0-} - 2)
	(\beta_{1+} + \beta_{1-} - 2)}
	{(z_1 - 1)^2 (z_2 - 1)^2 (z_3 - 1)^2} \biggr\} 
	 \frac{dz_1 \, dz_2 \,dz_3}{4 \lambda_{\infty}}, 
	\notag \\
W_{1, 1}(z)
	&= - \frac{z(z - \beta_{0+})(z - \beta_{0-})
	(z - \beta_{1+})(z - \beta_{1-})}
	{2 \lambda_{\infty} (z^2 - 1)^4} \, dz,  
\end{align*}
\begin{align*}	
F_0
&= \frac{(\lambda_0 + \lambda_1 + \lambda_{\infty} )^2}{2} 
\log{(\lambda_0 + \lambda_1 + \lambda_{\infty} )}
+ \frac{(\lambda_0 - \lambda_1 + \lambda_{\infty})^2}{2} 
\log{(\lambda_0 - \lambda_1 + \lambda_{\infty} )} \\
&\quad
	+ \frac{(\lambda_0 + \lambda_1 - \lambda_{\infty})^2}{2} 
	\log{(\lambda_0 + \lambda_1 - \lambda_{\infty})}
	+ \frac{(\lambda_0 - \lambda_1 - \lambda_{\infty})^2}{2} 
	\log{(\lambda_0 - \lambda_1 - \lambda_{\infty})} \notag \\
&\quad
	- 2 {\lambda_0}^2 \log{(2 \lambda_0)}
	- 2 {\lambda_1}^2 \log{(2 \lambda_1)} 
	- 2 {\lambda_{\infty}}^2 \log{(2 \lambda_{\infty})} ,  
	\notag \\
F_1
	&= - \frac{1}{12} \log{\left( 
	\frac{\Lambda}{\lambda_0 \lambda_1 \lambda_{\infty}} \right)}, \\
F_2
	&= \frac{1}{960 {{\lambda_0}^2 {\lambda_1 }^2 \lambda_{\infty}}^2 
	{\Lambda}^2} \{
	({\lambda_0}^2 + {\lambda_1}^2) {\lambda_{\infty}}^{10}
	- (4 {\lambda_0}^4 + 23 {\lambda_0}^2 {\lambda_1}^2 
	+ 4 {\lambda_1}^4) {\lambda_{\infty}}^8 \\
	&\quad
		+ 2({\lambda_0}^2 + {\lambda_1}^2)
			(3 {\lambda_0}^4 + 8 {\lambda_0}^2 {\lambda_1}^2 
			+ 3 {\lambda_1}^4) {\lambda_{\infty}}^6
		- 2(2 {\lambda_0}^8 - 11 {\lambda_0}^6 {\lambda_1}^2 
		+ 74 {\lambda_0}^4 {\lambda_1}^4  
		\notag  \\
	&\quad
		- 11 {\lambda_0}^2 {\lambda_1}^6 + 2 {\lambda_1}^8) 
		{\lambda_{\infty}}^4
		+ ({\lambda_0}^2 - {\lambda_1}^2)^2 ({\lambda_0}^2 + {\lambda_1}^2)
		({\lambda_0}^4 - 22 {\lambda_0}^2 {\lambda_1}^2 
		+ {\lambda_1}^4) {\lambda_{\infty}}^2  
			\notag \\
	&\quad
		+ {\lambda_0}^2 {\lambda_1}^2 
		({\lambda_0}^2 - {\lambda_1}^2)^4 \}. \notag 
\end{align*}
Since
\begin{equation}
Q_0(x) = \frac{{\lambda_{\infty}}^2 x^2 
- ({\lambda_{\infty}}^2 + {\lambda_0}^2 - {\lambda_1}^2)x 
+ {\lambda_0}^2}{x^2(1-x)^2}, \quad
Q_{\infty}(x) = 
\frac{{\lambda_{0}}^2 x^2 
- ({\lambda_{\infty}}^2 + {\lambda_0}^2 - {\lambda_1}^2)x 
+ {\lambda_\infty}^2}{x^2(1-x)^2},
\end{equation}
we have
\begin{equation}
\Sing (P) = \Sing_2(P) = \{0,1,\infty\},
\quad
B = B_1 = \{\beta_{0+}, \beta_{0-}, 
\beta_{1+}, \beta_{1-}, \beta_{\infty+}, \beta_{\infty-}\}.
\end{equation}
Therefore we choose 
\begin{align}
\label{eq:Gauss_D}
D(z; \underline{\nu})
&= [z]
 - \nu_{0+} [\beta_{0+}] - \nu_{0-} [\beta_{0-}]
- \nu_{1+} [\beta_{1+}] - \nu_{1-} [\beta_{1-}]
- \nu_{\infty+} [\beta_{\infty+}] - \nu_{\infty-} [\beta_{\infty-}]
\end{align}
as the divisor for the quantization, 
where the parameters $\nu_{0\pm}$, $\nu_{1\pm}$ 
and $\nu_{\infty\pm}$ satisfy
\begin{equation} \label{eq:relation-Gauss-nu}
\nu_{0+} + \nu_{0-} + \nu_{1+} + \nu_{1-} 
+ \nu_{\infty+} + \nu_{\infty-} = 1. 
\end{equation}
We may easily verify that $C_{\beta_{j\pm}} = \pm 2 \lambda_j$ 
for $j \in \{0,1,\infty \}$.
Then, Theorem \ref{thm:WKB-Wg,n} gives the quantum curve of the Gauss curve
(quantum Gauss curve):
\begin{equation}
\label{eq:Gauss_eq(d/dx)}
	\left( \hbar^2 \frac{d^2}{dx^2} + q(x,\hbar) \hbar \frac{d}{dx}
	+ r(x,\hbar) \right) \psi = 0,
\end{equation}
where 
$q(x,\hbar) = q_0(x) + \hbar q_1(x)$, 
$r(x,\hbar) = r_0(x) + \hbar r_1(x) + r_2(x)$ and
\begin{align*}
q_0(x) & = 0, \quad
q_1(x) = \frac{1 - \nu_{0+} - \nu_{0-}}{x}
		+ \frac{1 - \nu_{1+} - \nu_{1-}}{x - 1}, \\
r_0(x) &= - \frac{{\lambda_{\infty}}^2 x^2 - ({\lambda_{\infty}}^2 +
 {\lambda_0}^2 - {\lambda_1}^2)x + {\lambda_0}^2}{x^2 (1 - x)^2}, \\
r_1(x) &=  - \frac{(\nu_{0+} - \nu_{0-}) \lambda_0}{x^2 (x - 1)}
		+ \frac{(\nu_{1+} - \nu_{1-}) \lambda_1}{x(x - 1)^2}
		+ \frac{(\nu_{\infty+} - \nu_{\infty-}) \lambda_{\infty}}{x(x - 1)}, \\
r_2(x) &=  - \frac{\nu_{0+} \nu_{0-}}{x^2 (x - 1)}
		+ \frac{\nu_{1+} \nu_{1-}}{x(x - 1)^2}
		+ \frac{\nu_{\infty+} \nu_{\infty-}}{x(x - 1)}.
\end{align*}
The equation \eqref{eq:Gauss_eq(d/dx)} has three regular singular 
points at $0$, $1$ and $\infty$, and hence, 
it must be equivalent to the Gauss hypergeometric equation. 
Indeed, \eqref{eq:Gauss_eq(d/dx)} is related to
the standard form of the Gauss hypergeometric equation
\begin{equation} \label{eq:standard-Gauss-operator}
\left\{ x(1-x) \frac{d^2}{dx^2} + \bigl( \gamma - (\alpha+\beta+1) x \bigr) 
\frac{d}{dx} - \alpha \beta \right\} \phi = 0
\end{equation}
by 
\begin{equation}
\alpha = \frac{\hat{\lambda}_0 + \hat{\lambda}_1 + \hat{\lambda}_\infty}{\hbar} 
+ \frac{1}{2}, \quad
\beta = \frac{\hat{\lambda}_0 + \hat{\lambda}_1 - \hat{\lambda}_\infty}{\hbar} 
+ \frac{1}{2},  \quad
\gamma = \frac{2 \hat{\lambda}_0}{\hbar} + 1
\end{equation}
and the gauge transform 
\begin{equation} \label{eq:gauge-tramsform-Gauss}
\psi = x^{\frac{\hat{\lambda}_0}{\hbar} + \frac{\nu_{0+}+\nu_{0-}}{2}} 
(x-1)^{\frac{\hat{\lambda}_1}{\hbar} + \frac{\nu_{1+}+\nu_{1-}}{2}}  \phi.
\end{equation}
Here we used the notation 
\begin{equation} \label{eq:lambda-hat-Gauss}
\hat{\lambda}_j = \lambda_j - \frac{\hbar \nu_j}{2},\quad
\nu_j = \nu_{j+} - \nu_{j-} \quad (j \in \{0,1,\infty\}).
\end{equation}
We will investigate this equation in more detail in 
\S \ref{sec:Voros-vs-TR}.

\begin{rem} \label{rem:symmetry-Gauss}
The Gauss curve has a symmetry. 
That is, the symplectic transformations 
\begin{equation}
(x,y) \mapsto (x_j, y_j) := 
\begin{cases}
(1-x,-y) & \text{for $j=1$}, \\
(1/x, -x^2 y) & \text{for $j=\infty$},
\end{cases}
\end{equation}
yield the permutation 
\begin{equation} \label{eq:symmetry-Gauss}
(\lambda_0,\lambda_1,\lambda_\infty) \mapsto  
\begin{cases}
(\lambda_1,\lambda_0,\lambda_\infty) & \text{for $j=1$}, \\
(\lambda_\infty,\lambda_1,\lambda_0) & \text{for $j=\infty$},
\end{cases}
\end{equation}
of parameters in the Gauss curve \eqref{eq:Gauss_P(x,y)}.
Since the symplectic transform preserves the free energy (cf. \cite[\S7]{EO}), 
this proves that the free energy is invariant under any permutation 
of the parameters $\lambda_0,\lambda_1,\lambda_\infty$. 

Moreover, this symmetry of the Gauss curve is lifted to that of 
quantum Gauss curve \eqref{eq:Gauss_eq(d/dx)}. 
Namely, the parameters in the quantum Gauss curve are permuted as 
\begin{equation} \label{eq:symmetry-Gauss}
(\lambda_0,\lambda_1,\lambda_\infty, 
\nu_{0\pm}, \nu_{1\pm}, \nu_{\infty\pm}) \mapsto  
\begin{cases}
(\lambda_1,\lambda_0,\lambda_\infty,
\nu_{1\pm}, \nu_{0\pm}, \nu_{\infty\pm}) & \text{for $j=1$}, \\
(\lambda_\infty,\lambda_1,\lambda_0,
\nu_{\infty\pm}, \nu_{1\pm}, \nu_{0\pm}) & \text{for $j=\infty$},
\end{cases}
\end{equation}
under the coordinate transform 
$x \mapsto x_j$ ($j\in\{1, \infty \}$) given above.
This will induce a certain symmetry property of 
the Voros coefficients 
(see Theorem \ref{thm:main-theorem-in-part2} (iv)).  

\end{rem}

\subsubsection{Quantum degenerate Gauss curve}
\label{subsection:quantum-d-Gauss}

In this subsection we consider a spectral curve defined by
\begin{equation}
\label{eq:confluente:classic}
P(x, y) = x(x-1)^2 y^2 - {\lambda_{\infty}}^2 x - {\lambda_1}^2  + {\lambda_{\infty}}^2 = 0,
\end{equation}
which has a parametrization
\begin{equation}
\label{eq:confluente:xandy:1}
\left\{
\begin{array}{rcl}
x(z)  &=& \dfrac{{\lambda_1}^2 - {\lambda_{\infty}}^2}{z^2 - {\lambda_{\infty}}^2},
\\[2ex]
y(z) &=& - \dfrac{z (z^2 - {\lambda_{\infty}}^2)}{z^2 - {\lambda_1}^2}.
\end{array}
\right.
\end{equation}
Here we impose 
\begin{equation}
\lambda_1, \lambda_{\infty} \neq 0,  
\quad \text{and} \quad 
\lambda_{\infty} \pm \lambda_1 \neq 0
\end{equation}
for $\underline{\lambda} = (\lambda_1, \lambda_\infty)$
so that our assumptions are satisfied.
For this curve, we have
$R = \{0, \infty\} ~(=R^{\ast})$ and $\bar{z} = -z$.
First few correlation functions and free energies are
given explicitly as follows:
\begin{align*}
W_{0, 3}(z_1, z_2, z_3)
&=
- \frac{{\lambda_{\infty}}^2{\lambda_1}^2 }
{2({\lambda_1}^2 - {\lambda_{\infty}}^2)} 
\frac{dz_1 dz_2 dz_3}{{z_1}^2 {z_2}^2 {z_3}^2}, \\
 W_{1,1}(z_1) &=
\frac{1}{16 (\lambda_{\infty}^2 - \lambda_1^2)}
\frac{({z_1}^2 - {\lambda_{\infty}}^2) ({z_1}^2 - {\lambda_1}^2)}{{z_1}^4} dz_1, 
\end{align*}
\begin{align*}
F_0(\underline{\lambda}) & = 
\lambda_\infty^2 \log\left( 
\frac{\lambda_\infty^2 - \lambda_1^2}{4\lambda_\infty^2} \right)
+ 2\lambda_\infty \lambda_1 \log\left( 
\frac{\lambda_\infty + \lambda_1}{\lambda_\infty - \lambda_1} \right)
+ \lambda_1^2 \log\left( 
\frac{\lambda_\infty^2 - \lambda_1^2}{4\lambda_1^2} \right), \\
F_1(\underline{\lambda}) & = 
- \frac{1}{12} \log \left( 
\frac{(\lambda_\infty^2 - \lambda_1^2)^2}
{\lambda_\infty \lambda_1}
\right), \quad
F_2(\underline{\lambda}) = 
\frac{\lambda_\infty^6 - 17 \lambda_\infty^4 \lambda_1^2 
- 17 \lambda_\infty^2 \lambda_4^2 + \lambda_1^6}
{960 \lambda_\infty^2 \lambda_1^2 (\lambda_\infty^2 - \lambda_1^2)^2}.
\end{align*}
Theorem \ref{thm:WKB-Wg,n} implies that, 
for the choice 
\begin{equation}
D(z; \underline{\nu}) = [z] 
- \nu_{1 +}[\beta_{1+}] 
- \nu_{1 -}[\beta_{1-}]
- \nu_{\infty +}[\beta_{\infty+}]
- \nu_{\infty -}[\beta_{\infty-}]
\end{equation}
(where $\beta_{1\pm}=\pm \lambda_1$, 
$\beta_{\infty\pm} = \mp \lambda_\infty$, 
and the parameters satisfy 
$\nu_{1+} + \nu_{1-} + \nu_{\infty+} + \nu_{\infty-} =1$) 
of the divisor, we obtain the following quantum curve:
\begin{equation}
\label{eq:confluente:quantum}
\left( \hbar^2 \frac{d^2}{dx^2} + q(x,\hbar) \hbar \frac{d}{dx}
	+ r(x,\hbar) \right) \psi = 0
\end{equation}
with
\begin{align*}
q(x,\hbar) &= \hbar \left(\frac{1}{x} 
+ \frac{1 - \nu_{1+} - \nu_{1-}}{x-1} \right), \\
r(x,\hbar) & = 
- \frac{{\lambda_{\infty}}^2 x + {\lambda_1}^2 
- {\lambda_{\infty}}^2}{x(x-1)^2} 
+ \hbar \left( 
\frac{ (\nu_{1+} - \nu_{1-}) \lambda_1}{x(x-1)^2}
+\frac{ (\nu_{\infty+} - \nu_{\infty-}) \lambda_{\infty}}{x(x-1)} 
\right) \\ 
&  \quad
+ \hbar^2 \left(
\frac{\nu_{1+}\nu_{1-}}{x(x-1)^2}
+ \frac{\nu_{\infty+} \nu_{\infty-}}{x(x-1)} 
\right).
\end{align*}
This equation is a special case (i.e. $\lambda_0=0$, $\nu_{0\pm}=0$) 
of the quantum Gauss curve \eqref{eq:Gauss_eq(d/dx)}, 
and hence, we call \eqref{eq:confluente:quantum} 
the degenerate Gauss equation.
We distinguish these equations because the geometry of the 
spectral curves (or WKB-theoretic structure) are different. 

\subsubsection{Quantum Kummer curve}
\label{subsection:quantum-Kummer}

Let us consider the Kummer curve defined by
\begin{equation}
\label{eq:Kummer_P(x,y)}
P(x, y) = 4x^2 y^2 - 
\left( x^2 + 4 \lambda_{\infty} x + 4 {\lambda_0}^2 \right) = 0 
\qquad (\lambda_0 \ne 0, \lambda_0 \pm \lambda_\infty \ne 0).
\end{equation}
A rational parameterization of this curve is 
\begin{equation}
\label{eq:Kummer_parameterization}
\begin{cases}
	\displaystyle
	x = x(z) = \sqrt{{\lambda_{\infty}}^2 - {\lambda_0}^2} 
	\left( z + \frac{1}{z} \right) - 2 \lambda_{\infty} 
			= \frac{ \sqrt{{\lambda_{\infty}}^2 - {\lambda_0}^2} \, 
			(z - \beta_{0 +})(z - \beta_{0 -})}{z}, \\[10pt]
	\displaystyle
	y = y(z) = \frac{z^2 - 1}{2(z - \beta_{0 +})(z - \beta_{0 -})} ,
\end{cases}
\end{equation}
with $\beta_{0 \pm} = (\lambda_{\infty} \pm \lambda_0) / 
\sqrt{{\lambda_{\infty}}^2 - {\lambda_0}^2}$. 
We also set $\beta_{\infty+}=0$ and $\beta_{\infty-}=\infty$
for the consistency of our presentation.
The set of ramification points is given by $R = \{1, -1\} ~(=R^{\ast})$, 
and $\overline{z} = 1/z$.
 
Let $W_{g, n}$ and $F_g$ be the correlation functions 
and the free energies computed from 
\eqref{eq:Kummer_parameterization} respectively. 
Few of them are computed as
\begin{align*}
W_{0, 3}(z_1, z_2, z_3) 
&= - \frac{dz_1 \, dz_2 \,dz_3}
{2 \sqrt{{\lambda_{\infty}}^2 - {\lambda_0}^2}} 
\left\{ \frac{\beta_{0 +} + \beta_{0 -} + 2}
{(z_1 + 1)^2 (z_2 + 1)^2 (z_3 + 1)^2} 
- \frac{\beta_{0 +} + \beta_{0 -} - 2}
{(z_1 - 1)^2 (z_2 - 1)^2 (z_3 - 1)^2} \right\},
 \\
W_{1, 1}(z) 
&= - \frac{z^2(z - \beta_{0 +})(z - \beta_{0 -})}
{\sqrt{{\lambda_{\infty}}^2 - {\lambda_0}^2} (z^2 - 1)^4} \, dz,  
\end{align*}
\begin{align*}
F_0(\lambda_0, \lambda_{\infty}) 
&= \frac{(\lambda_{\infty} + \lambda_0)^2}{2} 
\log{(\lambda_{\infty} + \lambda_0)} 
+ \frac{(\lambda_{\infty} - \lambda_0)^2}{2} 
\log{(\lambda_{\infty} - \lambda_0)}  \\
&\quad - 2 {\lambda_0}^2 \log{(2 \lambda_0)} 
- \frac{3}{2} ({\lambda_{\infty}}^2 - {\lambda_0}^2), \nonumber \\
F_1(\lambda_0, \lambda_{\infty}) 
&= - \frac{1}{12} \log{\left( \frac{{\lambda_{\infty}}^2 
- {\lambda_0}^2}{\lambda_0} \right)}, 
\quad F_2(\lambda_0, \lambda_{\infty}) 
= \frac{{\lambda_{\infty}}^4 - 10 
{\lambda_{\infty}}^2 {\lambda_0}^2 - 7 {\lambda_0}^4}
{960 {\lambda_0}^2 ({\lambda_{\infty}}^2 - {\lambda_0}^2)^2}.
\end{align*}
Let us describe the quantization of (\ref{eq:Kummer_P(x,y)}). 
For the divisor 
\begin{equation}
\label{eq:Kummer_D}
	D(z; \underline{\nu})
	= [z] 
	- \nu_{\infty+} [\beta_{\infty+}] 
	- \nu_{\infty-} [\beta_{\infty-}] 
	- \nu_{0+} [\beta_{0 +}] - \nu_{0-} [\beta_{0 -}]
\end{equation}
(where the parameters satisfy 
$\nu_{\infty+} + \nu_{\infty-} + \nu_{0+} + \nu_{0-} = 1$),
Theorem \ref{thm:WKB-Wg,n} shows that 
the formal series 
$\psi(x;\underline{\nu}, \hbar) = \varphi (z(x); \underline{\nu}, \hbar)$ 
(cf. \eqref{eq:wave-z}) is a WKB solution of the following equation 
(quantum Kummer curve)
\begin{equation}
\label{eq:Kummer_eq(d/dx)}
\left( \hbar^2 \frac{d^2}{dx^2} + q(x,\hbar) \hbar \frac{d}{dx}
	+ r(x,\hbar) \right) \psi = 0
\end{equation}
with
\begin{align*}
q(x,\hbar) &= \hbar \, \frac{\nu_{\infty+} + \nu_{\infty-}}{x}, \\
r(x,\hbar) & = 
- \frac{x^2 + 4{\lambda_\infty} x + {\lambda_{0}}^2}{4x^2} 
+ \hbar \left( 
\frac{\nu_{\infty+} - \nu_{\infty-}}{2x} 
 + \frac{(\nu_{0+} - \nu_{0-}) \lambda_0}{x^2}
\right) 
+ \hbar^2 \, 
\frac{\nu_{0+} \nu_{0-}}{x^2}.
\end{align*}
This equation is equivalent to 
the Kummer differential equation 
(i.e., the confluent hypergeometric equation) 
via a certain gauge transformation.
Note also that the quantum curve \eqref{eq:Kummer_eq(d/dx)} 
is closely related to the quantum differential equation associated 
with the degree 1 hypersurface in ${\mathbb C}{\mathbb P}^1$. 
A realization of the equation as a quantum curve 
has already done in \cite[\S4.1]{FIMS}.

\subsubsection{Quantum Legendre curve}
\label{subsection:quantum-Legendre}

This subsection is devoted to study the Legendre curve
defined by
\begin{equation}
\label{eq:legendre:classic}
P(x, y):= (x^2-1)y^2 - \lambda_\infty^2 = 0
\qquad (\lambda_\infty \neq 0).
\end{equation}
A rational parametrization of this curve is
\begin{equation}
\left\{
\begin{array}{rcl}
x(z)  &=& \dfrac{1}{2} \left(z + \dfrac{1}{z}\right),
\\[+.7em]
y(z) &=& \dfrac{2\lambda_\infty z}{z^2 - 1}.
\end{array}
\right.
\end{equation}
For this curve, we have
$R = \{1, -1\} ~(=R^{\ast})$,
and the conjugate map $\bar{z} = 1/z$. 
Examples of correlation functions and free energies are:
\begin{align*}
W_{1,1}(z) &= - \frac{z}{2\lambda_\infty (z^2-1)^2} dz,
\quad
W_{0, 3}(z_1, z_2, z_3) = 0,\\
W_{1, 2}(z_1, z_2) &=
\frac{{z_1}^2 {z_2}^2 + {z_1}^2 + {z_2}^2 + 4 z_1 z_2 + 1}
{4\lambda_\infty^2 ({z_1}^2-1)^2 ({z_2}^2 - 1)^2} dz_1 dz_2,
\quad
W_{2,1}(z)
= - \frac{{z}^5 + 7 {z}^3 + z} {8\lambda_\infty^3 (z^2- 1)^4} dz,
\end{align*}
\begin{align*}
 F_0 &= - \lambda_\infty^2 \log 2,
\quad
 F_1 = - \frac{1}{4} \log \lambda_\infty,
\quad
 F_2 = -\frac{1}{64 \lambda_\infty^2},
\quad
 F_3 = - \frac{1}{256 \lambda_\infty^4}.
\end{align*}
Set $\beta_{\infty+}=0$, $\beta_{\infty-}=\infty$
and choose a divisor 
\begin{equation}
D(z; \underline{\nu}) = [z] 
- \nu_{\infty+} [\beta_{\infty+}] 
- \nu_{\infty-} [\beta_{\infty-}]
\qquad (\nu_{\infty+} + \nu_{\infty-} = 1).
\end{equation}
Then, the associated quantum curve by Theorem \ref{thm:WKB-Wg,n} 
is given by
\begin{equation}
\label{eq:legendere:quantum}
\left\{
\hbar^2 \frac{d^2}{dx^2} + \hbar^2 \frac{2 x}{x^2-1} \frac{d}{dx}
- \frac{\lambda_\infty^2}{x^2-1}
+ \hbar \frac{\lambda_\infty (\nu_{\infty+} - \nu_{\infty-})}{x^2-1}
+ \hbar^2 \frac{\nu_{\infty+} \nu_{\infty-}}{x^2-1}
\right\} \psi = 0.
\end{equation}
This equation \eqref{eq:legendere:quantum} has two simple-pole type
turning points at $x = \pm 1$ which correspond 
to the ramification points $z = \pm 1$ of the spectral curve.

\subsubsection{Quantum Bessel curve}
\label{subsection:quantum-Bessel}

Let us consider the Bessel curve defined by
\begin{equation}
\label{Bessel_P(x,y)}
	P(x, y) = y^2 - \frac{x  + 4\lambda_0^2}{4x^2}= 0 
	\qquad (\lambda_0 \ne 0).
\end{equation}
A rational parameterization of this curve is given by
\begin{equation}
\label{eq:Bessel_parameterization}
\begin{cases}
	\displaystyle
	x = x(z) = 4\lambda_0^2 (z^2 - 1), \\[5pt]
	\displaystyle
	y = y(z) = \frac{z}{4\lambda_0 (z^2 - 1)}.
\end{cases}
\end{equation}
For this curve, we have $R = \{0, \infty \}$. 
Note that $\infty \in R$ is ineffective, and hence 
$R^\ast = \{0 \}$ (cf.\,\cite[Proposition 2.7]{IKT-part1}).
The conjugate map is given by $\overline{z} = -z$. 
Few of $W_{g,n}$ and $F_g$ are computed as
\begin{eqnarray*}
W_{0,3}(z_1,z_2,z_3) & = & \frac{dz_1 dz_2 dz_3}
{2\lambda_0  z_1^2 z_2^2 z_3^2},
\quad
W_{1,1}(z) = -\frac{z^2 -1}{16\lambda_0 z^4} \, dz,\\
W_{1,2}(z_1,z_2) & = &
\frac{z_1^4 z_2^4 - 6 z_1^4 z_2^2 - 6 z_1^2 z_2^4
+ 5z_1^4 + 3 z_1^2 z_2^2 + 5 z_2^4}{32\lambda_0^2 z_1^6 z_2^6} 
\, dz_1 dz_2, \\
W_{2,1}(z) & = & 
- \frac{9z^6-107z^4+203z^2-105}{1024\lambda_0^3 z^{10}} \, dz, 
\end{eqnarray*}
\begin{equation*}
F_0 = 3 \lambda_0^2
+ \lambda_0^2  \log\left( - \frac{1}{16 \lambda_0^2} \right), \quad
F_1  =  - \frac{1}{24} \log\left( - \frac{1}{16 \lambda_0^2} \right), \quad
F_2 = \frac{1}{960 \lambda_0^2}, \quad
F_3 = - \frac{1}{16128 \lambda_0^4}.
\end{equation*}
Let us take a divisor
\begin{equation}
\label{eq:Bessel_D}
D(z; \underline{\nu}) = [z] 
- \nu_{0+}[\beta_{0+}] 
- \nu_{0-}[\beta_{0-}] 
- \nu_{\infty} [\infty] \qquad
(\nu_{0+} + \nu_{0-} + \nu_{\infty} = 1),
\end{equation}
where $\beta_{0\pm} = \pm 1$.  
Then, the quantization of the Bessel curve 
via Theorem \ref{thm:WKB-Wg,n} is given by
\begin{equation}
\label{eq:Bessel_eq}
	\left\{ \hbar^2 \frac{d^2}{dx^2}
	+ \hbar^2 \frac{1-\nu_{0+}-\nu_{0-}}{x} \frac{d}{dx}
	- \frac{x+4\lambda_0^2}{4x^2}
	+ \hbar \frac{\lambda_0(\nu_{0+}-\nu_{0-})}{x^2}
	+ \hbar^2 \frac{\nu_{0+} \nu_{0-}}{x^2} \right\} \psi = 0.
\end{equation}
This equation has a regular singular point at $x=0$ and
an irregular singular point at $x=\infty$.
We can verify that the equation \eqref{eq:Bessel_eq}
is equivalent to the Bessel equation
via a certain change of variables
and a gauge transform.
Note also that the special case $\nu_{0+} = \nu_{0-} = 0$ 
of the equation has already constructed as 
a quantum curve in \cite[\S4.1]{FIMS}; 
this example appears as the quantum differential equation 
for the equivariant ${\mathbb C}{\mathbb P}^1$.

\subsubsection{Quantum Whittaker curve}
\label{subsection:quantum-Whittaker}

In this subsection, we study the Whittaker curve
defined by
\begin{equation}
\label{eq:whittaker:classic}
P(x, y):= xy^2 - \frac{1}{4} x + \lambda_\infty = 0
\qquad (\lambda_\infty \neq 0).
\end{equation}
We parametrize this curve by
\begin{equation}
\label{eq:whittaker:xandy:1}
\left\{
\begin{array}{rcl}
x(z)  &=& - \dfrac{4\lambda_\infty}{z^2- 1} = 
2\lambda_\infty \left(\dfrac{1}{z + 1} - \dfrac{1}{z-1}\right),
\\
y(z) &=& \dfrac{1}{2} z.
\end{array}
\right.
\end{equation}
We find that $R = \{0, \infty\} ~(=R^{\ast})$ and $\bar{z} = - z$.
First few correlation functions and free energies are given as follows.
\begin{align*}
 W_{1,1}(z) &= 
 -\frac{(z^2 - 1)^2 }{32 \lambda_\infty {z}^4} dz,
\quad
W_{0, 3}(z_1, z_2, z_3) = -\frac{dz_1 dz_2 dz_3}
{4 \lambda_\infty  {z_1}^2{z_2}^2{z_3}^2}, 
\\
W_{1, 2}(z_1, z_2) &=
\frac{5{z_1}^4 + 3 {z_1}^2 {z_2}^2 + 5 {z_2}^4
-12 {z_1}^2 {z_2}^4 -12 {z_1}^4 {z_2}^2 + 10 {z_1}^4 {z_2}^4 
+ {z_1}^6 {z_2}^6}{128 \lambda_\infty^2 {z_1}^6 {z_2}^6} dz_1 dz_2,
\\
W_{2, 1} (z) &=
\frac{-9 {z}^{12}+22 {z}^{10}-103 {z}^8+372
{z}^6-583 {z}^4+406{z}^2-105}{8192 \lambda_\infty ^3 {z}^{10}}dz,
\end{align*}
\begin{align*}
F_0 &= -\frac{3}{2} \lambda_\infty^2 
+ \lambda_\infty^2 \log (4\lambda_\infty), \quad
F_1 = - \frac{1}{6}\log \lambda_\infty, \quad
F_2 = - \frac{1}{120 \lambda_\infty^2}.
\end{align*}
To quantize the Whittaker curve, we use a divisor
\begin{equation}
D(z; \underline{\nu}) := [z] 
- \nu_{\infty+} [\beta_{\infty+}] 
- \nu_{\infty-} [\beta_{\infty-}]
\qquad (\nu_{\infty+} + \nu_{\infty-} = 1),
\end{equation}
where $\beta_{\infty\pm}=\pm1$.
The quantum curve is given by (cf.\,Theorem \ref{thm:WKB-Wg,n})
\begin{equation}
\label{eq:whittaker:quantum}
\left\{
\hbar^2 \frac{d^2}{dx^2}
+ \hbar^2 \frac{1}{x} \frac{d}{dx}
- \frac{x - 4\lambda_\infty}{4x} 
-\hbar \frac{\nu_{\infty+} - \nu_{\infty-}}{2x}
\right\} \psi = 0.
\end{equation}

\subsubsection{Quantum Weber curve}
\label{subsection:quantum-Weber}

Computations presented in this subsection 
have been done in \cite{Takei17} and \cite{IKT-part1}. 
We discuss the Weber curve 
\begin{equation}
\label{Weber_P(x,y)}
P(x, y) := y^2 - \frac{x^2}{4} + \lambda_\infty = 0
\quad (\lambda_\infty \neq 0)
\end{equation}
with a parametrization
\begin{equation}
\label{Weber_parameterization}
\begin{cases}
\displaystyle
x = x(z) = \sqrt{\lambda_\infty} \left(z + \frac{1}{z}\right),\\
\displaystyle
y = y(z) = \frac{\sqrt{\lambda_\infty}}{2} \left(z - \frac{1}{z}\right).
\end{cases}
\end{equation}
The ramification points are given by $R = \{+1, -1 \} ~(=R^{\ast})$,
and the conjugation map becomes $\overline{z} = 1/z$.
Correlation functions and free energies for lower 
$g, n$ are given as follows.
\begin{align*}
W_{0, 3}(z_1, z_2, z_3)
&= \frac{1}{2 \lambda_\infty} \left\{ 
\frac{1}{(z_1 + 1)^2 (z_2 + 1)^2 (z_3 + 1)^2}
- \frac{1}{(z_1 - 1)^2 (z_2 - 1)^2 (z_3 - 1)^2} \right\} 
dz_1 \, dz_2 \,dz_3, \notag\\
W_{1, 1}(z) &= - \frac{z^3}{\lambda_\infty (z^2 - 1)^4} \, dz ,
\qquad
W_{2, 1}(z) = - \frac{21(z^{11}+ 3 z^9 + z^7)}
{\lambda_\infty^3 (z^2 - 1)^{10}} dz,
\end{align*}
\begin{align*}
F_0(\lambda) = - \frac{3}{4} \lambda_\infty^2
 + \frac{1}{2} \lambda_\infty^2 \log{\lambda_\infty},
\quad F_1(\lambda_\infty) = - \frac{1}{12} \log{\lambda_\infty},
\quad F_2(\lambda) = - \frac{1}{240 \lambda_\infty^2}.
\end{align*}
We set $\beta_{\infty+} = \infty$ and $\beta_{\infty-} = 0$, and
choose
\begin{equation} \label{Weber_D}
D(z; \underline{\nu}) := [z] 
- \nu_{\infty+} [\beta_{\infty+}] 
- \nu_{\infty-} [\beta_{\infty-}]
\qquad (\nu_{\infty+} + \nu_{\infty-} = 1)
\end{equation}
as the divisor for the quantization. 
The quantum curve for \eqref{Weber_P(x,y)} 
constructed by Theorem \ref{thm:WKB-Wg,n} becomes
\begin{equation}
\label{Weber_eq}
\left\{\hbar^2 \frac{d^2}{dx^2} -
\left( \frac{x^2}{4} - {\lambda_\infty} \right) 
- \hbar \frac{\nu_{\infty+} - \nu_{\infty-}}{2}
\right\}  \psi = 0.
\end{equation}

\subsubsection{Quantum degenerate Bessel curve}
\label{subsection:quantum-d-Bessel}

We discuss the degenerate Bessel curve defined by
\begin{equation}
\label{d-Bessel_P(x,y)}
P(x, y) := y^2 - \frac{1}{x} = 0
\end{equation}
with a parametrization
\begin{equation}
\label{d-Bessel_parameterization}
\begin{cases}
\displaystyle
x = x(z) = z^2,\\
\displaystyle
y = y(z) = \frac{1}{z}.
\end{cases}
\end{equation}
The ramification points are given by $R = \{0, \infty \}$, 
but $\infty$ is ineffective; i.e., $R^{\ast} = \{ 0\}$.
The conjugation map is $\overline{z} = - z$.
We choose
$D(z; \underline{\nu}) := [z] - [\infty]$ 
as the divisor for the quantization. 
The quantum curve constructed by Theorem \ref{thm:WKB-Wg,n} becomes
\begin{equation}
\label{degenerate-Bessel_eq}
\left\{\hbar^2 \frac{d^2}{dx^2} 
+ \hbar^2 \frac{1}{x} \frac{d}{dx}
- \frac{1}{x} \right\}  \psi = 0.
\end{equation}
The quantum curve result for this spectral curve 
was established by \cite[Theorem 8]{DN16}, 
where the curve \eqref{d-Bessel_P(x,y)} is called the Bessel curve.  
Also, this example appears as the quantum differential equation 
for the (non-equivariant) ${\mathbb C}{\mathbb P}^1$, 
as discussed in \cite{FIMS}.

\subsubsection{Quantum Airy curve}
\label{subsection:quantum-Airy}

As the last example, let us consider the Airy curve defined by
\begin{equation}
\label{Airy_P(x,y)}
P(x, y) := y^2 - x = 0 
\end{equation}
with a parametrization
\begin{equation}
\label{d-Bessel_parameterization}
\begin{cases}
\displaystyle
x = x(z) = z^2,\\
\displaystyle
y = y(z) = z.
\end{cases}
\end{equation}
For this spectral curve, we have
$R = \{0, \infty\}$ and $\overline{z} = -z$.
We choose the divisor $D(z) = [z] - [\infty]$ and 
the associated quantum curve of
the Airy curve is given in \cite{Zhou12} by 
\begin{equation}
\left\{ \hbar^2 \frac{d^2}{dx^2} - x \right\} \psi = 0.
\end{equation}

\begin{table}[t]
\begin{center}
\begin{tabular}{cc}\hline
 & $Q(x, \underline{\lambda}, \underline{\nu}; \hbar)$
\\\hline\hline
\parbox[c][3.5em][c]{0em}{}
Gauss (\S\ref{subsection:quantum-Gauss})
&
\begin{minipage}{.6\textwidth}
\begin{center}
$
\dfrac{{\hat{\lambda}_\infty}^2 x^2 - 
({\hat{\lambda}_\infty}^2 + {\hat{\lambda}_0}^2 - {\hat{\lambda}_1}^2)x 
+ {\hat{\lambda}_0}^2}{x^2 (x-1)^2}
- \hbar^2 \dfrac{x^2 -  x + 1}{4 x^2(x-1)^2}$
\end{center}
\end{minipage}
\\\hline
\parbox[c][3.5em][c]{0em}{}
Degenerate Gauss
(\S\ref{subsection:quantum-d-Gauss})
&
\begin{minipage}{.6\textwidth}
\begin{center}
$\dfrac{{\hat{\lambda}_\infty}^2 x + {\hat{\lambda}_1}^2 - {\hat{\lambda}_\infty}^2}{x(x-1)^2} - \hbar^2\dfrac{x^2-x+1}{4 x^2(x-1)^2}$\\
\end{center}
\end{minipage}
\\\hline
\parbox[c][3.5em][c]{0em}{}
Kummer (\S\ref{subsection:quantum-Kummer})
&
\begin{minipage}{.6\textwidth}
\begin{center}
$\dfrac{x^2 + 4 \hat{\lambda}_\infty x + 4 {\hat{\lambda}_0}^2}{4x^2} - \dfrac{\hbar^2}{4x^2}$\\
\end{center}
\end{minipage}
\\\hline
\parbox[c][3.53em][c]{0em}{}
Legendre (\S\ref{subsection:quantum-Legendre})
&
$\dfrac{{\hat{\lambda}_\infty}^2}{x^2-1} 
- \hbar^2\dfrac{x^2 + 3}{4 (x^2-1)^2}$
\\\hline
\parbox[c][3.5em][c]{0em}{}
Bessel (\S\ref{subsection:quantum-Bessel})
&
$\dfrac{x + {\hat{\lambda}_0}^2}{4x^2} - \dfrac{\hbar^2}{4x^2}$
\\\hline
\parbox[c][3.5em][c]{0em}{}
Whittaker (\S\ref{subsection:quantum-Whittaker})
& $\dfrac{x - 4 \hat{\lambda}_\infty}{4x} - \dfrac{\hbar^2}{4x^2}$
\\\hline
\parbox[c][3.5em][c]{0em}{}
Weber (\S\ref{subsection:quantum-Weber})
& $\dfrac{1}{4}x^2 - \hat{\lambda}_\infty$
\\\hline
\parbox[c][3.5em][c]{0em}{}
Degenerate Bessel (\S\ref{subsection:quantum-d-Bessel})
& $\dfrac{1}{x} - \dfrac{\hbar^2}{4x^2}$
\\\hline
\parbox[c][3.5em][c]{0em}{}
Airy (\S\ref{subsection:quantum-Airy})
& $x$
\\\hline
\end{tabular}
\caption{
SL-forms
$\{(\hbar d/dx)^2 - Q(x, \underline{\lambda}, \underline{\nu}; \hbar)\}
\varphi = 0$
of quantum curves of spectral curves.
In this table $\hat{\lambda}_j = \lambda_j - \hbar\nu_j/2$ 
and $\nu_{j} = \nu_{j+} - \nu_{j-}$ ($j \in  \{0, 1,\infty \}$).
For all cases,
the leading term of the potential is given by
$Q_{\text{cl}}(x)$.}
\label{table:quantum-curve}
\end{center}
\end{table}

\begin{rem}
It is more convenient for the exact WKB analysis
if the second order linear differential equation in question is
represented in the so-called SL-form, i.e.,
the second order linear differential equation with 
no first order term. 
Any second order linear differential equation of the form 
\eqref{eq:Gauss_eq(d/dx)} 
can be reduced to an SL-form
\begin{equation}
\label{eq:2nd-ODE:SL}
\left\{\hbar^2 \frac{d^2}{dx^2}
- Q(x, \hbar) \right\}\varphi = 0, 
\quad
Q(x, \hbar)
 := \frac{1}{4} q(x, \hbar)^2 - r(x, \hbar)
+ \frac{\hbar}{2} \frac{\partial}{\partial x} q(x, \hbar)
\end{equation}
by the gauge transformation
$\varphi = \exp \left(
\frac{1}{2\hbar}
\int^x q(x, \hbar) dx\right) \psi$
of the unknown function.
Table \ref{table:quantum-curve} shows the SL-form of 
quantum curves discussed in this subsection.
\end{rem}

\section{Free energies and Voros coefficients for 
the confluent family of hypergeometric quantum curves}
\label{sec:Voros-vs-TR}

\subsection{Definition of Voros coefficients}
\label{section:def-Voros}

Here we briefly recall the notion of Voros coefficients 
(see \cite{DDP93, Takei08, IN14} etc.\,for more details). 

Let 
\begin{equation}
\psi(x,\hbar) = \exp\left( \sum_{m = -1}^{\infty} 
\hbar^m \int^x S_m(x) \, dx \right)
\end{equation}
be the WKB solution (cf. \cite[\S2]{KT98}) 
of a second order equation
\begin{equation}
\label{eq:2nd-ODE-2}
\left\{
\hbar^2 \frac{d^2}{dx^2} + q(x, \hbar) \hbar \frac{d}{dx} + r(x, \hbar)
\right\}\psi = 0.
\end{equation}
Here we assume that the functions 
$q(x,\hbar) = q_0(x) + \cdots$ and 
$r(x,\hbar) = r_0(x) + \cdots$ in \eqref{eq:2nd-ODE-2}
are polynomials in $\hbar$ with coefficients which are rational in $x$ 
(thus the all quantum curves appeared in
\S \ref{section:example-quantum-curves} are of the above form). 
It is known that the logarithmic derivative 
\begin{equation}
S(x,\hbar) = \frac{d}{dx} \log \psi(x,\hbar) 
= \sum_{m=-1}^\infty \hbar^m S_m(x)
\end{equation}
of the WKB solution is a formal series (which is divergent in general) 
whose coefficients $S_m(x)$ are meromorphic on the 
algebraic curve $\Sigma$ defined by $y^2 + q_0(x) y + r_0(x) = 0$
(i.e., classical limit of the equation \eqref{eq:2nd-ODE-2}).
Let $b_1, b_2 \in \Sigma$ be preimages of singular points 
of the equation \eqref{eq:2nd-ODE-2} by the natural $2:1$ projection 
$\pi: \Sigma \to {\mathbb P}^1$, and take a path $\gamma$ 
on $\Sigma$ from $b_1$ to $b_2$ 
(for example, we may choose $\pi(b_1), \pi(b_2) \in \{0,1,\infty \}$ 
in the case of the quantum Gauss curve \eqref{eq:Gauss_eq(d/dx)}). 
Then, the Voros coefficient of \eqref{eq:2nd-ODE-2} for the path $\gamma$
is defined by 
\begin{equation} \label{eq:def-Voros-coeff}
V_\gamma(\hbar) := \int_{\gamma} \left( S(x,\hbar) 
- \hbar^{-1} S_{-1}(x) - S_0(x) \right) \, dx 
= \sum_{m = 1}^{\infty} \hbar^m \int_{\gamma} S_m(x)\,dx.
\end{equation}

Although the Voros coefficients are divergent series of $\hbar$ 
in general, they are Borel summable in a generic situation
and enjoy interesting Stokes phenomena 
related to cluster transformation (see \cite{DDP93, IN14}). 
Moreover, it is known that the global behavior 
(monodoromy or connection matrices)  of solutions of 
\eqref{eq:2nd-ODE-2} is described by Borel resummed Voros coefficients
(\cite[\S3]{KT98}). Thus the Voros coefficients are crucially 
important in the theory of linear differential equations. 

In this paper, however, we are not going to discuss those 
analytic properties of Voros coefficients. 
We focus on the properties of coefficients of the 
Voros coefficients. In particular, we will investigate 
the relationship between the correlation functions or 
free energies of the topological recursion and the Voros coefficients
through the framework of quantum curves.
Precise statements are given in the next subsection.

\subsection{Main Theorem}
\label{subsec:main-theorem}

In this subsection we formulate our main results
for the quantum curves discussed in
\S \ref{section:example-quantum-curves}. 
Our theorem allows us to express the Voros coefficients 
of the quantum curves by the free energy with 
a certain parameter shift
(cf. Theorem \ref{thm:main-theorem-in-part2} (i)).  
As a corollary, we can find 
three term relations satisfied by the free energy 
(cf. Theorem \ref{thm:main-theorem-in-part2} (ii)),
together with explicit formulas 
for the coefficients of the free energy and 
the Voros coefficients 
(as well as Theorem \ref{thm:main-theorem-in-part2} (iii) and (iv)).
We will only give the proof of our results 
for the (quantum) Gauss curve considered in 
\S \ref{subsection:quantum-Gauss} 
since the other examples can be treated in a similar manner.

Here we formulate our main results more precisely. 

Let $(C)$ be one of the spectral curves constructed 
in \S \ref{section:example-quantum-curves};
that is, 
Gauss (\S \ref{subsection:quantum-Gauss}), 
degenerate Gauss (\S \ref{subsection:quantum-d-Gauss}), 
Kummer (\S \ref{subsection:quantum-Kummer}), 
Legendre (\S \ref{subsection:quantum-Legendre}), 
Bessel (\S \ref{subsection:quantum-Bessel}), 
Whittaker (\S \ref{subsection:quantum-Whittaker}), 
or Weber (\S \ref{subsection:quantum-Weber}) curves.
Denote by $(E)$ the associated quantum curve by 
Theorem \ref{thm:WKB-Wg,n}. We will not consider 
the degenerate Bessel (\S \ref{subsection:quantum-d-Bessel})  
and Airy (\S \ref{subsection:quantum-Airy}) 
in this section because they have only trivial Voros coefficients 
(i.e., no non-trivial relative homology class on $\Sigma$), 
and accordingly, the free energy $F_g$ vanishes for $g \ge 2$.
Recall that the spectral curve $(C)$ is (a parametrization of) 
the classical limit $\Sigma$ of the quantum curve $(E)$.

Let $S$ be the set of singular points of $(E)$, 
which are not branch points of the spectral curve $(C)$.
The set $S$ is tabulated in 
Table \ref{table:singular-points-of-quantum-curves}. 
For each $j \in S$, there exist two preimages 
$\beta_{j+}$ and $\beta_{j-}$ of $j$ by 
the map $x(z)$ which appeared in the definition 
of the corresponding spectral curves 
(note that this notation is consistent 
with those given in 
\S \ref{subsection:quantum-Gauss} -- 
\S \ref{subsection:quantum-Airy}). 
Let $\gamma_j$ be a path on the Riemann surface $\Sigma$
defined as the image by $x(z)$ of a path from 
$\beta_{j-}$ to $\beta_{j+}$ on $z$-plane. 
(On $x$-plane, $\gamma_j$ starts from the point $j$ 
and come back to $j$ but with a different sheet of 
the Riemann surface $\Sigma$.)
Then, we define the Voros coefficient $V^{(j)}$ of 
the quantum curve $(E)$ for the path $\gamma_j$ by
\begin{equation} \label{eq:Voros-coeff-E} 
V^{(j)}(\underline{\lambda}, \underline{\nu}; \hbar)
= \int_{\gamma_j} \Bigl( S(x,\hbar) 
- \hbar^{-1} S_{-1}(x) - S_0(x) \Bigr) dx 
\quad (j \in S).
\end{equation}
Here $S(x,\hbar)$ is the logarithmic derivative of the WKB solution 
of the quantum curve $(E)$. 
Actually, $V^{(j)}$ depends only on the end and initial points 
of $\gamma_j$ since $S_m(x) dx$ for $m \ge 1$ 
(or the correlation functions $W_{g,n}$ for $2g-2+n \ge 1$) 
has no residue at branch points.
We also note that, since the SL-form of the quantum curves  
in Table \ref{table:quantum-curve} depends only on 
the difference $\nu_j$ of the parameter $\nu_{j\pm}$, 
so do the Voros coefficients. 
Thus we have used the symbol 
$\underline{\lambda} = (\lambda_{j})_{j \in S}$, 
$\underline{\nu} = (\nu_{j})_{j \in S}$
for the tuple of parameters contained in 
the quantum curve $(E)$ in the expression \eqref{eq:Voros-coeff-E}.

\begin{table}[t]
\begin{center}
{\renewcommand{\arraystretch}{1.3}
\begin{tabular}{|c||c|c|c|c|c|c|c|c|c|} \hline
    $(E)$ 
    & Gauss 
    & Degenerate Gauss 
    & Kummer
    & Legendre
    & Bessel
    & Whittaker
    & Weber
    \\ \hline
    $S$ 
    & $\{0,1,\infty \}$ 
    & $\{1,\infty \}$ 
    & $\{0,\infty \}$ 
    & $\{\infty \}$ 
    & $\{0\}$ 
    & $\{\infty \}$ 
    & $\{\infty \}$ 
    \\ \hline
\end{tabular}}
\end{center} \vspace{-.7em}
\caption{The list of singular points 
(which are not branch points of the spectral curve) 
of quantum curves. 
Note that, some of regular singular point 
(e.g., $x=0$ for the quantum degenerate Gauss curve)
are eliminated from the above list. 
This is because those points are
``turnings point of simple-pole type" 
in the sense of \cite{Ko2}, 
which play similar roles as turning points in the WKB analysis. 
Hence, we will not associate the Voros coefficient for those points.}
\label{table:singular-points-of-quantum-curves}
\end{table}

On the other hand, the topological recursion associates 
the free energy 
\begin{equation} \label{eq:total-free-energy-C}
F(\underline{\lambda}; \hbar)
= \sum_{g = 0}^{\infty} \hbar^{2g - 2} F_g(\underline{\lambda})
\end{equation}
for the spectral curve $(C)$ as we have reviewed in 
\S \ref{section:free-energy-part-2}.
Our main result gives a formula which relates 
the free energy of $(C)$ and the Voros coefficients of $(E)$. 
We derive three-terms relations satisfied by the free energy, 
and moreover, obtain explicit expressions of the 
free energies and Voros coefficients as a corollary of 
these relations. The precise statement is formulated as follows.

\begin{thm} \label{thm:main-theorem-in-part2}
\begin{itemize}
\item[{\rm (i)}] 
The Voros coefficient of the quantum curve $(E)$
and the free energy of the spectral curve $(C)$ are related as follows.
\begin{equation} \label{eq:V-and-F-general}
V^{(j)}(\underline{\lambda}, \underline{\nu}; \hbar)
= F\left(\underline{\hat{\lambda}} 
+ \frac{\hbar}{2} \delta_j; \hbar\right)
- F\left(\underline{\hat{\lambda}} 
- \frac{\hbar}{2} \delta_j; \hbar\right)
-\frac{1}{\hbar}\frac{\partial F_0}{\partial\lambda_j}
+ \frac{1}{2} \sum_{k \in S} 
\nu_k \frac{\partial^2 F_0}{\partial\lambda_j \lambda_k} 
\quad (j \in S).
\end{equation}
Here the symbol $\underline{\hat{\lambda}} + \alpha \delta_j$
means the parameter shift of $\hat{\lambda}_j$ by $\alpha$:
\begin{equation}
\underline{\hat{\lambda}} + \alpha \delta_j 
:= (\hat{\lambda}_k + \delta_{kj} \alpha)_{k \in S}. 
\end{equation}
(Recall $\hat{\lambda}_j = \lambda_j - \hbar \nu_j/2$ 
with $\nu_j = \nu_{j+} - \nu_{j-}$ ($j=0,1,\infty$) 
as we have introduced in \eqref{eq:lambda-hat-Gauss}.)

\item[{\rm (ii)}] 
The free energy for the spectral curve $(C)$ 
satisfies the following three-term difference equation.
\begin{align} \label{eq:three-term-general}
F(\underline{\lambda} + \hbar \delta_j; \hbar)
- 2 F(\underline{\lambda}; \hbar)
+ F(\underline{\lambda} - \hbar \delta_j; \hbar)
= 
\log \Lambda(\underline{\lambda}) + R_j(\underline{\lambda}; \hbar)
\quad (j \in S).
\end{align}
Here $\Lambda(\underline{\lambda})$ and $R_j(\underline{\lambda}; \hbar)$ in 
the right hand-side of \eqref{eq:three-term-general} 
are given as follows:

\begin{itemize}
\item[$\bullet$] 
For Gauss curve (\S \ref{subsection:quantum-Gauss}):
\[
\Lambda(\underline{\lambda}) = 
(\lambda_0+\lambda_1+\lambda_\infty)(\lambda_0+\lambda_1-\lambda_\infty)
(\lambda_0-\lambda_1+\lambda_\infty)(\lambda_0-\lambda_1-\lambda_\infty),
\]
\[
R_j(\underline{\lambda};\hbar) = 
- 2\log(2\lambda_j) - \log(2\lambda_j+\hbar) - \log(2\lambda_j - \hbar)
\qquad (j \in \{0,1,\infty \}).
\]

\item[$\bullet$]
For degenerate Gauss curve (\S \ref{subsection:quantum-d-Gauss}):
\[
\Lambda(\underline{\lambda}) = 
(\lambda_1+\lambda_\infty)^2(\lambda_1-\lambda_\infty)^2,
\]
\[
R_j(\underline{\lambda};\hbar) = 
- 2\log(2\lambda_j) - \log(2\lambda_j+\hbar) - \log(2\lambda_j - \hbar)
\qquad (j \in \{1,\infty \}).
\]

\item[$\bullet$]
For Kummer curve (\S \ref{subsection:quantum-Kummer}): 
\[
\Lambda(\underline{\lambda}) = 
(\lambda_\infty+\lambda_0)(\lambda_\infty-\lambda_0),
\]
\[
R_j(\underline{\lambda};\hbar) = 
\begin{cases}
\displaystyle 
- 2\log(2\lambda_0) - \log(2\lambda_0+\hbar) - \log(2\lambda_0 - \hbar)
& ~~(j = 0), \\[+.5em]
0 & ~~(j=\infty).
\end{cases} 
\]

\item[$\bullet$]
For Legendre curve (\S \ref{subsection:quantum-Legendre}): 
\[
\Lambda(\lambda_\infty) = 4 \lambda_\infty^4, \quad
R_\infty(\lambda_\infty;\hbar) = 
-2\log(2\lambda_\infty) - \log(2\lambda_\infty+\hbar) - \log(2\lambda_\infty - \hbar).
\]

\item[$\bullet$] 
For Bessel curve (\S \ref{subsection:quantum-Bessel}):
\[
\Lambda(\lambda_0) = \dfrac{1}{16}, \quad 
R_0(\lambda_0;\hbar) = 
- 2\log(2\lambda_0) - \log(2\lambda_0+\hbar) - \log(2\lambda_0 - \hbar).
\]

\item[$\bullet$] 
For Whittaker curve (\S \ref{subsection:quantum-Whittaker}): 
\[
\Lambda(\lambda_\infty) = \lambda_\infty^2,
\quad
R_{\infty}(\lambda_\infty;\hbar) = 0.
\]

\item[$\bullet$]
For Weber curve (\S \ref{subsection:quantum-Weber}):
\[
\Lambda(\lambda_\infty) = \lambda_\infty, 
\quad 
R_\infty(\lambda_\infty;\hbar) = 0.
\]

\end{itemize}

\item[{\rm (iii)}]
For $g \geq 2$, the $g$-th free energy of 
the spectral curve $(C)$
has the following expression.
\begin{itemize}
\item[$\bullet$] For Gauss curve (\S \ref{subsection:quantum-Gauss}):
\end{itemize} \vspace{-1.3em}
\begin{multline*}
F_{g}(\underline{\lambda}) = 
\frac{B_{2g}}{2g(2g-2)}
\left\{
\frac{1}{(\lambda_0 + \lambda_1 + \lambda_{\infty})^{2g-2}}
+ \frac{1}{(\lambda_0 - \lambda_1 + \lambda_{\infty})^{2g-2}}
+ \frac{1}{(\lambda_0 + \lambda_1 - \lambda_{\infty})^{2g-2}}
\right.
\\
\left.
+ \frac{1}{(\lambda_0 - \lambda_1 - \lambda_{\infty})^{2g-2}}
- \frac{1}{(2\lambda_0)^{2g-2}}
- \frac{1}{(2\lambda_1)^{2g-2}}
- \frac{1}{(2\lambda_{\infty})^{2g-2}}
\right\}. 
\end{multline*}
($F_0$ and $F_1$ for Gauss curve are given in \S \ref{subsection:quantum-Gauss}.)

\begin{itemize}
\item[$\bullet$] 
For degenerate Gauss curve (\S \ref{subsection:quantum-d-Gauss}):
\end{itemize} 
\[
F_g(\underline{\lambda})
= \dfrac{B_{2g}}{2g (2g-2)}
\left\{
\dfrac{2}{(\lambda_1 + \lambda_\infty)^{2g-2}}
+ \dfrac{2}{(\lambda_1 - \lambda_\infty)^{2g-2}}
- \dfrac{1}{(2 \lambda_1)^{2g-2}} 
- \dfrac{1}{(2 \lambda_\infty)^{2g-2}}
\right\}.
\]
($F_0$ and $F_1$ for degenerate Gauss curve are given in 
\S \ref{subsection:quantum-d-Gauss}.)

\begin{itemize}
\item[$\bullet$] 
For Kummer curve (\S \ref{subsection:quantum-Kummer}): 
\end{itemize}
\[
F_g(\underline{\lambda}) 
= \frac{B_{2g}}{2g(2g - 2)} 
 \left\{ \frac{1}{(\lambda_0+\lambda_{\infty})^{2g - 2}} 
      + \frac{1}{(\lambda_0-\lambda_{\infty} )^{2g - 2}} 
     - \frac{1}{(2 \lambda_0)^{2g - 2}} \right\}.
\]
($F_0$ and $F_1$ for Kummer curve are given in 
\S \ref{subsection:quantum-Kummer}.)

\begin{itemize}
\item[$\bullet$] 
For Legendre curve (\S \ref{subsection:quantum-Legendre}): 
\end{itemize}
\[
F_g(\lambda_\infty) 
= \frac{B_{2g}}{2g(2g-2)} \left\{ \frac{4}{\lambda_\infty^{2g-2}} 
 - \frac{1}{(2 \lambda_\infty)^{2g-2}} \right\}.
\]
($F_0$ and $F_1$ for Legendre curve are given in 
\S \ref{subsection:quantum-Legendre}.)

\begin{itemize}
\item[$\bullet$] 
For Bessel curve (\S \ref{subsection:quantum-Bessel}):
\end{itemize}
\[
F_g(\lambda_0) = 
- \frac{B_{2g}}{2g(2g-2)} \frac{1}{(2\lambda_0)^{2g - 2}}. 
\]
($F_0$ and $F_1$ for Bessel curve are given in 
\S \ref{subsection:quantum-Bessel}. 
The expression was obtained in \cite{IM}.)

\begin{itemize}
\item[$\bullet$] 
For Whittaker curve (\S \ref{subsection:quantum-Whittaker}): 
\end{itemize}
\[
F_g(\lambda_\infty) = 
\frac{B_{2g}}{2g(2g-2)} \frac{2}{\lambda_\infty^{2g - 2}}. 
\]
($F_0$ and $F_1$ for Whittaker curve are given in 
\S \ref{subsection:quantum-Whittaker}.)

\begin{itemize}
\item[$\bullet$] 
For Weber curve (\S \ref{subsection:quantum-Weber}):
\end{itemize}
\[
F_g(\lambda_\infty) = \dfrac{B_{2g}}{2g(2g - 2)} 
\dfrac{1}{\lambda_\infty^{2g-2}}.
\]
($F_0$ and $F_1$ for Weber curve are given in 
\S \ref{subsection:quantum-Weber}. 
The expression was obtained in \cite{HZ}; 
see also \cite{IM}.)

\bigskip
In these expressions, $B_{m}$ is the $m$-th Bernoulli number
given by \eqref{def:Bernoulli} in \S \ref{sec:bernoulli}.

\item[{\rm (iv)}]
The Voros coefficients \eqref{eq:Voros-coeff-E} 
for the quantum curve $(E)$ are explicitly given by 
\begin{equation}
V^{(j)}(\underline{\lambda}, \underline{\nu}; \hbar)
= \sum_{m = 1}^{\infty} \hbar^{m} 
V^{(j)}_{m}(\underline{\lambda}, \underline{\nu})
\quad (j \in S),
\end{equation}
where the coefficients 
$V^{(j)}_m(\underline{\lambda}, \underline{\nu})$
for $m \ge 1$ are given as follows.

\begin{itemize}
\item[$\bullet$]  
For quantum Gauss curve (\S \ref{subsection:quantum-Gauss}):
\end{itemize} \vspace{-1.em}
\begin{align*}
& V^{(0)}_m(\underline{\lambda}, \underline{\nu}) 
 = V^{(0)}_m(\lambda_0, \lambda_1, \lambda_\infty, 
\nu_0, \nu_1, \nu_\infty) \\
& = 
\frac{1}{m (m + 1)}
\left\{
\frac{B_{m + 1} ((\nu_0 + \nu_1 + \nu_{\infty} + 1)/2)}
{(\lambda_0 + \lambda_1 + \lambda_{\infty})^m}
+ \frac{B_{m + 1} ((\nu_0 - \nu_1 + \nu_{\infty} + 1)/2)}
{(\lambda_0 - \lambda_1 + \lambda_{\infty})^m}
\right. \nonumber \\ 
&
+ \frac{B_{m + 1} ((\nu_0 + \nu_1 - \nu_{\infty} + 1)/2)}
{(\lambda_0 + \lambda_1 - \lambda_{\infty})^m}
+ \frac{B_{m + 1} ((\nu_0 - \nu_1 - \nu_{\infty} + 1)/2)}
{(\lambda_0 - \lambda_1 - \lambda_{\infty})^m} 
\left.
- \frac{B_{m + 1}(\nu_0) + B_{m + 1}(\nu_0 + 1)}{(2\lambda_0)^m} 
\right\}
\end{align*}
gives the formula for $j=0$. The other Voros coefficients 
(for $j=1$ and $\infty$) are obtained from 
the above formula by permuting the parameters 
$\underline{\lambda}$ and $\underline{\nu}$: 
\[
V^{(1)}_m(\underline{\lambda}, \underline{\nu}) 
 =
V^{(0)}_m(\lambda_1, \lambda_0, \lambda_\infty, 
\nu_1, \nu_0, \nu_\infty), \qquad
V^{(\infty)}_m(\underline{\lambda}, \underline{\nu}) 
 = 
V^{(0)}_m(\lambda_\infty, \lambda_1, \lambda_0, 
\nu_\infty, \nu_1, \nu_0). 
\]
(These expressions were also obtained in \cite{ATT2}.)

\begin{itemize}
\item[$\bullet$]  
For quantum degenerate Gauss curve (\S \ref{subsection:quantum-d-Gauss}):
\end{itemize}
\begin{align*}
& V^{(1)}_m(\underline{\lambda}, \underline{\nu}) 
= V^{(1)}_m(\lambda_1, \lambda_\infty, \nu_1, \nu_\infty) \\
& = \frac{1}{m(m+1)}\left\{
\frac{2 B_{m + 1} ((\nu_1 + \nu_{\infty} + 1)/2)}
{(\lambda_1 + \lambda_{\infty})^m}
+ \frac{2 B_{m + 1} ((\nu_1 - \nu_\infty + 1)/2)}
{(\lambda_1 - \lambda_\infty)^m}
\right. \nonumber \\ 
 & \qquad
 \left.
- \frac{B_{m + 1}(\nu_1) + B_{m + 1}(\nu_1 + 1)}
{(2\lambda_1)^m} 
\right\}
\end{align*}
gives the formula for $j=1$. 
The other Voros coefficient for $\infty$ is 
obtained from the above formula by permuting the parameters:
\[
V^{(\infty)}_m(\lambda_1, \lambda_\infty, \nu_1, \nu_\infty)
=  
V^{(1)}(\lambda_\infty, \lambda_1, \nu_\infty, \nu_1).
\]

\begin{itemize}
\item[$\bullet$]  
For quantum Kummer curve (\S \ref{subsection:quantum-Kummer}): 
\end{itemize}
\begin{align*}
		V^{(0)}_m(\underline{\lambda}, \underline{\nu}) 
		&= \frac{1}{m(m+1)} 
			\bigg\{ \frac{B_{m+1}((\nu_{0} + \nu_{\infty} + 1)/2)}
			{(\lambda_0 + \lambda_{\infty})^{m}} 
		+ \frac{B_{m+1}((\nu_{0} - \nu_{\infty} + 1)/2)}
		{(\lambda_0 - \lambda_{\infty})^{m}} \\
		&\qquad 
		- \frac{B_{m+1}(\nu_{0}) + B_{m+1}(\nu_{0}+1)}
		{{(2 \lambda_0)}^{m}} \bigg\}, \notag \\
		\label{eq:Kummer_Voros(concrete-form)(infty)}
		V^{(\infty)}_m(\underline{\lambda}, \underline{\nu}) 
		&=  \frac{1}{m(m+1)} 
			\bigg\{ \frac{B_{m+1}((\nu_{0} + \nu_{\infty} + 1)/2)}
			{(\lambda_0 + \lambda_{\infty})^{m}} 
			- \frac{B_{m+1}((\nu_{0} - \nu_{\infty} + 1)/2)}
			{(\lambda_0 - \lambda_{\infty})^{m}} \bigg\}. 
\end{align*}
(These expressions were also obtained in \cite{ATT}.)

\begin{itemize}
\item[$\bullet$]  
For quantum Legendre curve (\S \ref{subsection:quantum-Legendre}): 
\end{itemize}
\[
V^{(\infty)}_m(\lambda_\infty, \nu_\infty) =  
\frac{1}{m(m+1)} \left\{  
\frac{4B_{m+1}((\nu_\infty + 1)/2)}{\lambda_\infty^m} 
- \frac{B_{m+1}(\nu_\infty) + B_{m+1}(\nu_\infty + 1)}{(2\lambda_\infty)^{m}}
\right\}.
\]
(This expression was also obtained in \cite{Ko3}.)

\begin{itemize}
\item[$\bullet$]  
For quantum Bessel curve (\S \ref{subsection:quantum-Bessel}):
\end{itemize}
\[
V^{(0)}_m(\lambda_0, \nu_0) =  
- \frac{B_{m+1}(\nu_0) + B_{m+1}(\nu_0 + 1)}{m(m+1) (2\lambda_0)^{m}}. 
\]
(The special case $\nu_0 = 0$ of the expression 
was obtained in \cite{AIT}.)

\begin{itemize}
\item[$\bullet$]  
For quantum Whittaker curve (\S \ref{subsection:quantum-Whittaker}): 
\end{itemize}
\[
V^{(\infty)}_m(\lambda_\infty, \nu_\infty) =  
\frac{2 B_{m+1} ( (\nu_\infty+1)/2)}{m(m+1)\lambda_{\infty}^{m}}.
\]
(This expression was obtained in \cite{KoT11}.)

\begin{itemize}
\item[$\bullet$]  
For quantum Weber curve (\S \ref{subsection:quantum-Weber}):
\end{itemize}
\[
V^{(\infty)}_m(\lambda_\infty, \nu_\infty) = 
\frac{B_{m + 1}\big( (\nu_\infty+1)/2 \big)}{m (m + 1) \lambda_\infty^{m}}. 
\]
(This expression was obtained in \cite{Takei17} and
\cite{IKT-part1};  
see also \cite{Takei08} for the special case $\nu_\infty=0$.)

\bigskip
In these expressions, $B_{m}(t)$ is 
the $m$-th Bernoulli polynomial
given by \eqref{def:BernoulliPoly} in \S \ref{sec:bernoulli}. 

\end{itemize}
\end{thm}

We will prove these statements in the rest of this section.

\subsection{Proof of main theorem (for quantum Gauss curve)}
\label{subsec:proof-gauss}

Here we give a proof of Theorem \ref{thm:main-theorem-in-part2}
for the case of (quantum) Gauss curve.  
Therefore, in the following proof, we focus on 
the Guass curve \eqref{eq:Gauss_parameterization}
and the associated quantum Gauss curve \eqref{eq:Gauss_eq(d/dx)} 
and use the notations in \S \ref{subsection:quantum-Gauss};
namely, the function $x(z)$ is given 
in \eqref{eq:Gauss_parameterization}, and the divisor 
$D(z;\underline{\nu})$ is given in \eqref{eq:Gauss_D}, 
and the set $S$ of singular points becomes $\{0,1, \infty \}$, etc. 
(We omit the proof for other examples 
considered in \S \ref{section:example-quantum-curves} 
because they can be treated in a similar manner.)

\subsubsection{Proof of Theorem \ref{thm:main-theorem-in-part2} (i)}

For the proof, we need 
\begin{lem} \label{lemma:Gauss_variation}
The following relations hold for $2g-2+n \ge 1$:
\begin{equation}
\label{eq:Gauss_variation}
\frac{\partial^n}{\partial \lambda_{j}^n}  F_g
=
\int_{\zeta_1=\beta_{j -}}^{\zeta_1=\beta_{j +}} \cdots 
\int_{\zeta_n=\beta_{j -}}^{\zeta_n=\beta_{j +}} 
W_{g, n}(\zeta_1, \ldots, \zeta_n) 
\qquad (j \in \{0,1,\infty \}).
\end{equation}
\end{lem}
\begin{proof}
Because
\begin{equation}
\frac{\partial y(z)}{\partial \lambda_j} \cdot dx(z)
- \frac{\partial x(z)}{\partial \lambda_j} \cdot dy(z)
= \int^{\zeta = \beta_{j+}}_{\zeta = \beta_{j-}} B(z, \zeta)
\end{equation}
holds for each $j \in \{0,1,\infty \}$, 
the variation formula \cite[Theorem 5.1]{EO} 
implies \eqref{eq:Gauss_variation}. 
\end{proof}
Note that the parameters $\lambda_j$ and points $\beta_{j\pm}$
in the other examples in \S \ref{section:example-quantum-curves}
are chosen so that the variation formula 
\eqref{eq:Gauss_variation} holds. 
Therefore, following proof is applicable to 
the other examples.  

\begin{lem}
\label{lem:Gauss_path_deformation}
Let $\omega(z)$ be a meromorphic differential whose poles 
are only on the set of ramification points of 
$x(z)$ in \eqref{eq:Gauss_parameterization}. 
Suppose also that $\omega(z)$ has no residue at those poles, 
and is anti-invariant under the conjugation map; 
that is, $\omega(\bar{z}) = - \omega(z)$. 
Then, we have
	\begin{equation}
		\int_{\beta_{j\epsilon}}^{\beta_{k\eta}} \omega(z) 
		= - \int_{\beta_{j(-\epsilon)}}^{\beta_{k(-\eta)}}  \omega(z) 
		= \frac{1}{2} \int_{\beta_{j\epsilon}}^{\beta_{j(-\epsilon)}} \omega(z) 
		+ \frac{1}{2}\int_{\beta_{k(-\eta)}}^{\beta_{k\eta}} \omega(z) 
				\label{eq:Gauss_path_deformation} 
\end{equation}
for any $j,k \in \{0,1,\infty \}$ and $\epsilon, \eta \in \{+,- \}$.
\end{lem}
This is proved by a straightforward computation and relations 
$\overline{\beta_{j\pm}}=\beta_{j\mp}$ ($j \in \{0,1, \infty\}$).

\bigskip
Now we derive the equality \eqref{eq:V-and-F-general} 
with the aid of Lemmas \ref{lemma:Gauss_variation} 
and \ref{lem:Gauss_path_deformation} and the results in 
\S \ref{subsection:quantum-Gauss} on the quantization of Gauss curve. 
Let us focus on the Voros coefficient 
$V^{(0)}(\underline{\lambda}, \underline{\nu};\hbar)$ 
for the singular point $j=0$. 
It follows from the definition of the path $\gamma_0$
and Theorem \ref{thm:WKB-Wg,n} that the Voros coefficient 
associated with the singular point $0$ is expressed 
in terms of $W_{g,n}$ as follows.
\begin{align}
V^{(0)}(\underline{\lambda}, \underline{\nu}; \hbar)
&= \int_{\beta_{0-}}^{\beta_{0+}}
\Bigl( S(x(z), \hbar) - \hbar^{-1} S_{-1}(x(z)) 
- S_0(x(z)) \Bigr) \frac{dx}{dz} \, dz \\
&= \sum_{m = 1}^{\infty} \hbar^m \int_{\beta_{0-}}^{\beta_{0+}}
\left\{ \sum_{\substack{2g - 2 + n = m \\ g \geq 0, \, n \geq 1}}
\frac{1}{n!} \frac{d}{dz} \int_{\zeta_1 \in D(z; \underline{\nu})}
\cdots \int_{\zeta_n \in D(z; \underline{\nu})}
W_{g, n}(\zeta_1, \ldots, \zeta_n) \right\} dz
\notag\\
&= \sum_{m = 1}^{\infty} \hbar^m
\sum_{\substack{2g - 2 + n = m \\ g \geq 0, \, n \geq 1}}
\frac{1}{n!}
\left(
\int_{\zeta_1 \in D(\beta_{0+}; \underline{\nu})} \cdots 
\int_{\zeta_n \in D(\beta_{0+}; \underline{\nu})}
\right.
\notag\\
&\qquad\qquad\qquad\qquad\qquad\qquad
\left.
-
\int_{\zeta_1 \in D(\beta_{0-}; \underline{\nu})} \cdots 
\int_{\zeta_n \in D(\beta_{0-}; \underline{\nu})}
\right)
W_{g, n}(\zeta_1, \ldots, \zeta_n).
\notag
\end{align}
For simplicity, let us introduce the notation 
\begin{equation}
\left\{ \int_{\gamma} \right\}^n 
W_{g,n}(\zeta_1, \ldots, \zeta_n)
= \int_{\zeta_1 \in \gamma} \cdots \int_{\zeta_n \in \gamma} 
W_{g,n}(\zeta_1, \ldots, \zeta_n).
\end{equation}
Since 
\begin{equation}
D(\beta_{0\pm}; \underline{\nu}) = 
\nu_{0\mp}([\beta_{0\pm}] - [\beta_{0\mp}]) +
\sum_{\substack{j \in \{1,\infty \} \\ \epsilon \in \{+,-\}}} 
\nu_{j\epsilon} ([\beta_{0\pm}] - [\beta_{j\epsilon}]) 
\end{equation}
(here we have used \eqref{eq:relation-Gauss-nu}), 
Lemmas \ref{lemma:Gauss_variation} 
and \ref{lem:Gauss_path_deformation} imply
\begin{equation}
\left\{\int_{D(\beta_{0\pm}; \underline{\nu})} \right\}^n 
W_{g,n}(\zeta_1,\dots,\zeta_n)   \hspace{+17.em}
\end{equation}
\vspace{-1.em}
\begin{align*}
& = \left\{ \left(\nu_{0\mp} 
+ \sum_{\substack{j \in \{1,\infty \} \\ \epsilon \in \{+,-\}}} 
\frac{\nu_{j\epsilon}}{2} \right) \int_{\beta_{0\mp}}^{\beta_{0\pm}} 
+ \sum_{\substack{j \in \{1,\infty \} \\ \epsilon \in \{+,-\}}} 
\frac{\nu_{j\epsilon}}{2} \int_{\beta_{j\epsilon}}^{\beta_{j(-\epsilon)}}
\right\}^n  W_{g,n}(\zeta_1,\dots,\zeta_n)  \\
& = \left\{ 
\frac{-\nu_0 \pm 1}{2} \int_{\beta_{0-}}^{\beta_{0+}} 
- \sum_{j \in \{1, \infty\}} \frac{\nu_j}{2} \int_{\beta_{j-}}^{\beta_{j+}} 
\right\}^n 
W_{g,n}(\zeta_1,\dots,\zeta_n) \\
& = \left\{ 
\frac{-\nu_0 \pm 1}{2} \frac{\partial}{\partial \lambda_0} 
- \sum_{j \in \{1, \infty\}} \frac{\nu_j}{2} 
\frac{\partial}{\partial \lambda_j} \right\}^n 
F_g(\underline{\lambda}).
\end{align*} 
Here we used \eqref{eq:relation-Gauss-nu} again, 
and set $\nu_j = \nu_{j+} - \nu_{j-}$ ($j \in \{ 0,1,\infty \}$) 
as we have introduced in \eqref{eq:lambda-hat-Gauss}.
Further computation shows
\begin{equation*}
\begin{split}
V^{(0)}(\underline{\lambda}, \underline{\nu}; \hbar) 
&= \sum_{m = 1}^{\infty} \hbar^m 
\sum_{\substack{2g - 2 + n = m \\ g \geq 0, \, n \geq 1}} 
\frac{1}{n!} \\ 
& 
\times \left[ 
\left\{ 
\frac{-\nu_0 + 1}{2} \frac{\partial}{\partial \lambda_0} 
- \sum_{j \in \{1, \infty\}} \frac{\nu_j}{2} 
\frac{\partial}{\partial \lambda_j} \right\}^n 
 - 
\left\{ 
\frac{-\nu_0 - 1}{2} \frac{\partial}{\partial \lambda_0} 
- \sum_{j \in \{1, \infty\}} \frac{\nu_j}{2} 
\frac{\partial}{\partial \lambda_j} \right\}^n 
 \right]F_g(\underline{\lambda}) \\
&= 
\sum_{\substack{ g \ge 0, \,\, n_1,n_2,n_3 \ge 0 \\ 2g-2+n_1+n_2+n_3 \ge 1}}
\frac{1}{n_1! n_2! n_3!} \\
& \times \left[ 
\left\{ \left( 
-\frac{\hbar \nu_0}{2} + \frac{\hbar}{2}
\right) \frac{\partial}{\partial \lambda_0} 
\right\}^{n_1} 
\left\{ -\frac{\hbar \nu_1}{2} \frac{\partial}{\partial \lambda_1} \right\}^{n_2} 
\left\{-\frac{\hbar \nu_\infty}{2}  \frac{\partial}{\partial \lambda_\infty} \right\}^{n_3} 
\hbar^{2g-2} F_g(\underline{\lambda}) \right. \\
& \left.
- \left\{ \left( 
-\frac{\hbar \nu_0}{2} - \frac{\hbar}{2}
\right) \frac{\partial}{\partial \lambda_0} 
\right\}^{n_1} 
\left\{ -\frac{\hbar \nu_1}{2} \frac{\partial}{\partial \lambda_1} \right\}^{n_2} 
\left\{-\frac{\hbar \nu_\infty}{2}  \frac{\partial}{\partial \lambda_\infty} \right\}^{n_3} 
\hbar^{2g-2} F_g(\underline{\lambda}) \right]
\\
&= 
F \left(\lambda_0 - \frac{\hbar\nu_0}{2} + \frac{\hbar}{2}, 
\lambda_1 - \frac{\hbar\nu_1}{2}, \lambda_\infty - \frac{\hbar\nu_\infty}{2}; \hbar \right)
\\ 
&  \quad 
- F \left(\lambda_0 - \frac{\hbar\nu_0}{2} - \frac{\hbar}{2}, 
\lambda_1 - \frac{\hbar\nu_1}{2}, \lambda_\infty - \frac{\hbar\nu_\infty}{2}; \hbar \right)
\\ 
& \qquad
- \frac{1}{\hbar} \frac{\partial F_0}{\partial \lambda_0} 
+ \frac{1}{2} \left( 
\nu_0 \frac{\partial}{\partial \lambda_0} 
+ \nu_1 \frac{\partial}{\partial \lambda_1} 
+ \nu_\infty \frac{\partial}{\partial \lambda_\infty} 
\right) \frac{\partial F_0}{\partial \lambda_0}.
\end{split}
\end{equation*}
Thus we have proved \eqref{eq:V-and-F-general} for $j=0$. 
The other equalities for $j=1$ and $j=\infty$ can be derived similarly.

\begin{rem} \label{remark:zeta-reg}
In the definition \eqref{eq:Voros-coeff-E} of the Voros coefficient, 
we subtracted the first two terms $\hbar^{-1}S_{-1}$ and $S_0$ 
because these terms are singular at end points of the path $\gamma_j$. 
However, a regularization procedure of divergent integral 
(see \cite{Voros-zeta} for example)
allows us to define the regularized Voros coefficient:
\begin{equation}
V^{(j)}_{\rm reg}(\underline{\lambda},\underline{\nu};\hbar) 
= 
\hbar^{-1}V^{(j)}_{-1}(\underline{\lambda},\underline{\nu}) 
+ V^{(j)}_0(\underline{\lambda},\underline{\nu})
+ V^{(j)}(\underline{\lambda},\underline{\nu};\hbar) 
\qquad (j \in \{0,1,\infty\} ), 
\end{equation}
where 
$V^{(j)}_{-1}(\underline{\lambda},\underline{\nu})$ and 
$V^{(j)}_0(\underline{\lambda},\underline{\nu})$ are obtained by solving 
\begin{equation} \label{eq:zeta-regularization-equation}
\frac{\partial^2}{\partial \lambda_j^2} V^{(j)}_{-1} = 
\int_{\gamma_j} \frac{\partial^2}{\partial \lambda_j^2} S_{-1}(x) \, dx, \quad
\frac{\partial}{\partial \lambda_j} V^{(j)}_{0} = 
\int_{\gamma_j} \frac{\partial}{\partial \lambda_j} S_{0}(x) \, dx 
\qquad (j \in \{0,1,\infty \}).
\end{equation}
Actually, we can verify that $\partial_{\lambda_j}^2 S_{-1}(x)$ and 
$\partial_{\lambda_j}S_0(x)$ are holomorphic at $x=j$ although  
$S_{-1}$ and $S_0$ are singular there.
Hence, the equations in \eqref{eq:zeta-regularization-equation} make sense
and we can find $V^{(j)}_{-1}$ and $V^{(j)}_{0}$. 
For example, in the case of quantum Gauss curve, 
the equations \eqref{eq:zeta-regularization-equation} 
for $j=0$ are solved explicitly by 
\begin{align}
V^{(0)}_{-1} 
& = 
(\lambda_0 + \lambda_1 + \lambda_{\infty})\log{(\lambda_0 + \lambda_1 + \lambda_{\infty})} 
(\lambda_0 - \lambda_1 + \lambda_{\infty})\log{(\lambda_0 - \lambda_1 + \lambda_{\infty})} \\
& \quad
+(\lambda_0 + \lambda_1 - \lambda_{\infty})\log{(\lambda_0 + \lambda_1 - \lambda_{\infty})}
(\lambda_0 - \lambda_1 - \lambda_{\infty})\log{(\lambda_0 - \lambda_1 - \lambda_{\infty})} \notag \\
& \quad 
- 4 \lambda_0 \log(2\lambda_0), \notag
\end{align}
\begin{align}
V^{(0)}_{0} &= - \frac{\nu_0}{2} \Bigl( 
\log{(\lambda_0 + \lambda_1 + \lambda_{\infty})} 
+ \log{(\lambda_0 - \lambda_1 + \lambda_{\infty})} 
+ \log{(\lambda_0 + \lambda_1 - \lambda_{\infty})} 
+ \log{(\lambda_0 - \lambda_1 - \lambda_{\infty})} \\
& \hspace{+3.7em} 
- 4 \log{(2 \lambda_0)} \Bigr) \notag \\
& \quad 
- \frac{\nu_1}{2} \Bigl( 
\log{(\lambda_0 + \lambda_1 + \lambda_{\infty})} 
- \log{(\lambda_0 - \lambda_1 + \lambda_{\infty})}
+ \log{(\lambda_0 + \lambda_1 - \lambda_{\infty})}
- \log{(\lambda_0 - \lambda_1 - \lambda_{\infty})} \Bigr) \notag \\
& \quad
- \frac{\nu_\infty}{2} \Bigl(
\log{(\lambda_0 + \lambda_1 + \lambda_{\infty})} 
+ \log{(\lambda_0 - \lambda_1 + \lambda_{\infty})}
- \log{(\lambda_0 + \lambda_1 - \lambda_{\infty})}
- \log{(\lambda_0 - \lambda_1 - \lambda_{\infty})}
\Bigr). \notag
\end{align}
Actually, we can verify that the regularized integrals are realized by
the correction terms 
\begin{equation}
V^{(j)}_{-1} = \frac{\partial F_0}{\partial \lambda_j}, \qquad
V^{(j)}_{0} = - \frac{1}{2} \left( 
\nu_0 \frac{\partial}{\partial \lambda_0} 
+ \nu_1 \frac{\partial}{\partial \lambda_1} 
+ \nu_\infty \frac{\partial}{\partial \lambda_\infty} 
\right) \frac{\partial F_0}{\partial \lambda_j} 
\qquad (j \in \{0,1,\infty \})
\end{equation}
in the right hand-side of the relation \eqref{eq:V-and-F-general}.
Thus we can conclude that the regularized Voros coefficient satisfies 
\begin{equation}
V^{(j)}_{\rm reg}(\underline{\lambda},\underline{\nu};\hbar) 
= 
F\left(\underline{\hat{\lambda}} 
+ \frac{\hbar}{2} \delta_j; \hbar\right)
- F\left(\underline{\hat{\lambda}} 
- \frac{\hbar}{2} \delta_j; \hbar\right) 
\qquad (j \in \{0,1,\infty \}).
\end{equation}
\end{rem}

\subsubsection{Proof of Theorem \ref{thm:main-theorem-in-part2} (ii)}

Using the contiguity relations for the Gauss hypergeometric equation
(see \S \ref{sec:contiguity-relation}), 
we can find difference equations (with respect to the parameter $\nu_j$) 
satisfied by the Voros coefficients.
\begin{lem}  \label{lem:Gauss_Voros-parameter}
The Voros coefficient $V^{(0)}$ (for the singular point $0$)
of the quantum Gauss curve \eqref{eq:Gauss_eq(d/dx)} 
satisfies the following relations. 
\begin{align}  
\label{eq:Gauss-Voros-contiguity-1}
&V^{(0)}(\underline{\lambda}, \nu_0, \nu_1, \nu_{\infty}; \hbar) 
- V^{(0)}(\underline{\lambda}, \nu_0-1, \nu_1-1, \nu_{\infty}; \hbar) \\
&\qquad 
= - \log \left[
\frac{\lambda_0^2}{\hat{\lambda}_0
(\hat{\lambda}_0 + {\hbar}/{2})} \,
\frac{(\hat{\lambda}_0+\hat{\lambda}_1
+\hat{\lambda}_\infty + {\hbar}/{2})}
{(\lambda_0+\lambda_1+\lambda_\infty)} \,
\frac{(\hat{\lambda}_0+\hat{\lambda}_1
- \hat{\lambda}_\infty + {\hbar}/{2})}
{(\lambda_0+\lambda_1-\lambda_\infty)}
\right], 
\notag \\
\label{eq:Gauss-Voros-contiguity-2}
&V^{(0)}(\underline{\lambda}, \nu_0, \nu_1, \nu_{\infty}; \hbar) 
- V^{(0)}(\underline{\lambda},\nu_0, \nu_1 - 1, \nu_{\infty}+1;\hbar)
\\
&\qquad
= 
- \log \left[
\frac{(\hat{\lambda}_0 + \hat{\lambda}_1 
- \hat{\lambda}_\infty + {\hbar}/{2})}
{(\lambda_0 + \lambda_1 - \lambda_\infty)} \,
\frac{(\lambda_0 - \lambda_1 + \lambda_\infty)}
{(\hat{\lambda}_0 - \hat{\lambda}_1 + \hat{\lambda}_\infty  
- {\hbar}/{2})}
\right],
\notag\\
\label{eq:Gauss-Voros-contiguity-3}
&V^{(0)}(\underline{\lambda},\nu_0, \nu_1, \nu_{\infty};\hbar) 
- V^{(0)}(\underline{\lambda}, \nu_0 - 1, \nu_1 , \nu_{\infty}+1; \hbar)
\\
&\qquad
=
- \log \left[ 
\frac{\lambda_0^2}
{\hat{\lambda}_0(\hat{\lambda}_0+{\hbar}/{2})} \,
\frac{(\hat{\lambda}_0 + \hat{\lambda}_1 
- \hat{\lambda}_\infty + {\hbar}/{2})}
{(\lambda_0 + \lambda_1 -\lambda_\infty)} \,
\frac{(\hat{\lambda}_0 - \hat{\lambda}_1 
- \hat{\lambda}_\infty + {\hbar}/{2})}
{(\lambda_0 - \lambda_1 -\lambda_\infty)} 
\right]. 
\notag 
\end{align}
\end{lem}
We will prove Lemma \ref{lem:Gauss_Voros-parameter} in 
\S \ref{subsection:proof-of-contiguity-Voros}. 
Here we derive \eqref{eq:three-term-general} for $j=0$
by using this result. 

In the proof $V^{(0)}(\underline{\lambda}, \underline{\nu};\hbar)$ 
is abbreviated as $V^{(0)}(\nu_0, \nu_1, \nu_\infty)$ 
since we are only interested in dependence on $\underline{\nu}$. 
Substituting $(\nu_0,\nu_1,\nu_\infty) = (1,0,0)$  
into \eqref{eq:Gauss-Voros-contiguity-1},
$(\nu_0,\nu_1,\nu_\infty) = (0,0,-1)$
into \eqref{eq:Gauss-Voros-contiguity-2}, and
$(\nu_0,\nu_1,\nu_\infty) = (0,0,-1)$
into \eqref{eq:Gauss-Voros-contiguity-3}, 
we have
\begin{align}
V^{(0)}(1,0,0) - V^{(0)}(0,-1,0) 
& =  
-  \log \left[ \frac{2\lambda_0}{2\lambda_0-\hbar}
\right],
\\
V^{(0)}(0,0,-1) - V^{(0)}(0,-1,0) 
& = 0,
\\
V^{(0)}(0,0,-1) - V^{(0)}(-1,0,0) 
& =
- \log \left[ \frac{2\lambda_0}{2\lambda_0+\hbar}
\right].
\end{align}
Combining these equalities, we have
\begin{equation} \label{eq:Gauss-Voros-difference-special}
V^{(0)}(1,0,0) - V^{(0)}(-1,0,0) = 
- \log \left[ \frac{4\lambda_0^2}
{(2\lambda_0+\hbar)(2\lambda_0-\hbar)}
\right].
\end{equation}
On the other hand, \eqref{eq:V-and-F-general} 
for $j=0$ with $(\nu_0,\nu_1,\nu_\infty) = (1,0,0)$
and $(\nu_0,\nu_1,\nu_\infty) = (-1,0,0)$ gives
\begin{align}
\label{eq:Gauss-V-and-F-special-1}
& V^{(0)}(1,0,0)
 = 
F\left( \lambda_0,\lambda_1,\lambda_\infty ;\hbar\right) 
- F\left( \lambda_0-\hbar,\lambda_1,\lambda_\infty ;\hbar\right) 
- \frac{1}{\hbar} \frac{\partial F_0}{\partial \lambda_0} 
+ \frac{1}{2} \frac{\partial^2 F_0}{\partial \lambda_0^2}, \\
\label{eq:Gauss-V-and-F-special-2}
& V^{(0)}(-1,0,0) 
 = 
F\left( \lambda_0+\hbar,\lambda_1,\lambda_\infty ;\hbar\right) 
- F\left( \lambda_0,\lambda_1,\lambda_\infty ;\hbar\right) 
- \frac{1}{\hbar} \frac{\partial F_0}{\partial \lambda_0} 
- \frac{1}{2} \frac{\partial^2 F_0}{\partial \lambda_0^2}.
\end{align} 
Then, the desired equality \eqref{eq:three-term-general} 
for $j=0$ follows immediately from 
\eqref{eq:Gauss-Voros-difference-special}, 
\eqref{eq:Gauss-V-and-F-special-1},
\eqref{eq:Gauss-V-and-F-special-2} 
and the explicit expression of $F_0$
given in \S \ref{subsection:quantum-Gauss}.

As we have seen in Remark \ref{rem:symmetry-Gauss}, 
the free energy is symmetric with respect to the parameters
$(\lambda_0, \lambda_1, \lambda_\infty)$. 
Therefore, \eqref{eq:three-term-general} for $j=0$ 
implies the other equalities for $j=1$ and $\infty$.
Thus we have proved Theorem \ref{thm:main-theorem-in-part2} (ii).

\subsubsection{Proof of Theorem \ref{thm:main-theorem-in-part2} (iii)}
\label{subsection:solving-difference-F}

Let $X_j$ ($j \in S$) be the shift operator defined by 
\begin{equation}\label{eq:Xj}
X_j = e^{- \hbar \partial_{\lambda_j}} 
\left( e^{\hbar \partial_{\lambda_j}} - 1 \right)^2 ~~:~~
F(\lambda_j) \mapsto  
F(\lambda_j+\hbar) - 2F(\lambda_j) + F(\lambda_j-\hbar), 
\end{equation} 
where 
$\partial_{\lambda_j} = {\partial}/{\partial \lambda_j}$.
Using this operator, the three term relation 
\eqref{eq:three-term-general} is written as 
\begin{align} 
\label{eq:Gauss_Xj}
X_j F\left(\underline{\lambda}; \hbar\right) 
&= \log{\Lambda}(\underline{\lambda}) 
- 2 \log{(2 \lambda_j)} - \log{(2 \lambda_j + \hbar)} - \log{(2 \lambda_j - \hbar)}. 
\end{align}
Our goal is to solve the difference equation 
(or invert the operator $X_j$) and obtain 
the explicit expression of the coefficients of $F$.
For this purpose, we use 

\begin{lem} \label{lem:partial-solution-of-difference-eq}
\begin{itemize}
\item[{\rm (i)}] 
Let 
$G(\underline{\lambda}; \hbar)  = 
\sum_{g = 0}^{\infty} \hbar^{2g-2} G_g(\underline{\lambda})$
be the formal series with the coefficients
\begin{align}
G_0(\underline{\lambda}) & 
= 
\frac{(\lambda_0+\lambda_1+\lambda_\infty)^2}{2} \log(\lambda_0+\lambda_1+\lambda_\infty)
+\frac{(\lambda_0-\lambda_1-\lambda_\infty)^2}{2} \log(\lambda_0-\lambda_1-\lambda_\infty)  \\
& 
+\frac{(\lambda_0+\lambda_1+\lambda_\infty)^2}{2} \log(\lambda_0+\lambda_1+\lambda_\infty)
+\frac{(\lambda_0+\lambda_1+\lambda_\infty)^2}{2} \log(\lambda_0+\lambda_1+\lambda_\infty)
\notag \\ 
& 
-3(\lambda_0^2+\lambda_1^2+\lambda_\infty^2), \notag
\end{align}
\begin{align}
G_1(\underline{\lambda}) & = 
- \frac{1}{12} \log \Lambda, \\
\label{eq:Gg-higher}
G_g(\underline{\lambda}) &  
= \frac{B_{2g}}{2g(2g-2)} 
\left\{\frac{1}{(\lambda_0+\lambda_1+\lambda_\infty)^{2g-2}} 
+\frac{1}{(\lambda_0-\lambda_1+\lambda_\infty)^{2g-2}} \right.  \\
& 
\left. 
+\frac{1}{(\lambda_0+\lambda_1-\lambda_\infty)^{2g-2}}
+\frac{1}{(\lambda_0-\lambda_1-\lambda_\infty)^{2g-2}} \right\} 
\qquad (g \ge 2). 
\notag 
\end{align}
Then, $G(\underline{\lambda}; \hbar)$ satisfies the following difference equation:
\begin{equation}
X_j G(\underline{\lambda};\hbar)  =  \log \Lambda
\qquad (j \in \{0,1,\infty \}).
\end{equation}

\item[{\rm (ii)}]
Let $H(\underline{\lambda}; \hbar)  = 
\sum_{g = 0}^{\infty} \hbar^{2g-2} H_g(\underline{\lambda})$
be the formal series with the coefficients
\begin{align}
H_{0}(\underline{\lambda}) & = 
-2 \lambda_0^2 \log(2\lambda_0)
-2 \lambda_1^2 \log(2\lambda_1)
-2 \lambda_\infty^2 \log(2\lambda_\infty)
+ 3 (\lambda_0^2 + \lambda_1^2 + \lambda_\infty^2), \\
H_{1}(\underline{\lambda}) & = 
\frac{1}{12} \log (\lambda_0 \lambda_1 \lambda_\infty ),  \\
\label{eq:Hg-higher}
H_g(\underline{\lambda}) & = 
- \frac{B_{2g}}{2g(2g-2)}\left\{ 
\frac{1}{(2\lambda_0)^{2g-2}}
+ \frac{1}{(2\lambda_1)^{2g-2}}
+ \frac{1}{(2\lambda_\infty)^{2g-2}} \right\} 
\qquad (g \ge 2).
\end{align}
Then, $H(\underline{\lambda}; \hbar)$ satisfies the following difference equation:
\begin{equation}
X_j H(\underline{\lambda}; \hbar) =  
- \log{(2 \lambda_j + \hbar)} - 2 \log{(2 \lambda_j)} - \log{(2 \lambda_j - \hbar)} 
\qquad (j \in \{0,1, \infty \}).
\end{equation}
\end{itemize}
\end{lem}

\begin{proof}
Firstly, we note that the terms $G_0$, $H_0$ are chosen so that the relations
\begin{equation}
\partial_{\lambda_j}^2 G_0(\underline{\lambda}) = \log \Lambda, 
\quad 
\partial_{\lambda_j}^2 H_0(\underline{\lambda}) = - 4 \log{(2 \lambda_j)} 
\end{equation}
hold for each $j \in \{0,1,\infty \}$. 
Therefore, the claim (i) (resp., (ii)) immediately follows from 
Proposition \ref{prop:difference-eq:sol} (i) 
(resp., Proposition \ref{prop:difference-eq:sol} (ii)) 
which will be proved in \S \ref{subsection:solving-difference-eq}.  
\end{proof}

Note that the formal series 
$F$ and $G + H$ have the same first two terms:
\begin{equation}
F_0 = G_0 + H_0, \quad
F_1 = G_1 + H_1.
\end{equation}
Lemma \ref{lem:partial-solution-of-difference-eq} 
shows that they satisfy the same difference equations; 
that is, 
\begin{equation}
X_j (F - (G + H)) = 0 \qquad (j \in \{0,1,\infty \}).
\end{equation}
Then, Proposition \ref{prop:almost-uniqueness-2} implies that 
\begin{equation} \label{eq:F-G-H-Gauss}
F - (G + H) = \sum_{g = 2}^{\infty} \hbar^{2g-2} 
\left( C_g 
+ \sum_{j \in S} C^{(j)}_{g} \lambda_j
+ \sum_{\substack{j,k \in S \\ j \ne k}} 
C^{(jk)}_g \lambda_j \lambda_k 
+ C^{(01\infty)}_g \lambda_0 \lambda_1 \lambda_\infty 
\right)
\end{equation}
holds. Here $C_g$, $C^{(j)}_{g}$, $C^{(jk)}_g$ 
and $C^{(01\infty)}_g$ are some complex numbers
which are independent of $\lambda_0, \lambda_1, \lambda_\infty$. 
To show that these constants vanish, we use the homogeneity of $F_g$'s.
\begin{lem}
The following relations hold 
for any nonzero complex number $t$.
\begin{eqnarray} 
\label{eq:homogenity-Fg-Gauss}
F_g(t \lambda_0, t \lambda_1, t \lambda_\infty) & = & 
t^{2-2g} F_g(\lambda_0, \lambda_1, \lambda_\infty) \qquad (g \ge 2), \\
G_g(t \lambda_0, t \lambda_1, t \lambda_\infty) & = & 
t^{2-2g} G_g(\lambda_0, \lambda_1, \lambda_\infty) \qquad (g \ge 2), \\
\label{eq:homogenity-Hg-Gauss}
H_g(t \lambda_0, t \lambda_1, t \lambda_\infty) & = & 
t^{2-2g} H_g(\lambda_0, \lambda_1, \lambda_\infty) \qquad (g \ge 2).
\end{eqnarray}
\end{lem}
The first equality \eqref{eq:homogenity-Fg-Gauss} 
is a consequence of \cite[Theorem 5.3]{EO}. 
Therefore, we can conclude that 
the right hand-side of \eqref{eq:F-G-H-Gauss} 
must vanish due to the homogeneity. 
Consequently, we obtain the equality $F = G + H$ 
which proves the desired equality for the free energy of the Gauss curve.
This completes the proof of Theorem \ref{thm:main-theorem-in-part2} (iii).

\begin{rem}
The other results in Theorem \ref{thm:main-theorem-in-part2} (iii) 
are easily obtained from the difference equation \eqref{eq:three-term-general}
by a similar method presented here. 
In particular, the formal series in Proposition \ref{prop:difference-eq:sol} (i) 
appears when we compute the action of inverse operator of $X_j$
on (each factor of) $\log \Lambda(\underline{\lambda})$ 
appearing in \eqref{eq:three-term-general}. 
Similarly, the other formal series in Proposition \ref{prop:difference-eq:sol} (ii)
is the inverse of $R_{j}(\underline{\lambda};\hbar)$. 
Although the solution of the difference equation is not unique, 
we can fix the ambiguity with the aid of homogeneity, as we have done here. 
\end{rem}

\subsubsection{Proof of Theorem \ref{thm:main-theorem-in-part2} (iv)}

Here we derive the explicit formula of the coefficient of 
the Voros coefficient $V^{(0)}(\underline{\lambda}, \underline{\nu}; \hbar)$
of the quantum Gauss curve 
by using the relation \eqref{eq:V-and-F-general} 
between the Voros coefficient and the free energy.
Firstly, applying the difference operator $e^{\hbar \partial_{\lambda_0}} - 1$
to the both sides of the relation \eqref{eq:V-and-F-general} for $j=0$, 
we have 
\begin{multline} \label{eq:Gauss-Voros-X0}
\left( e^{\hbar \partial_{\lambda_0}} - 1 \right) 
V^{(0)}(\underline{\lambda}, \underline{\nu}; \hbar) 
= 
e^{\hbar \partial_{\lambda_0} /2} 
X_0 F\left(\hat{\lambda}_0, \hat{\lambda}_1, \hat{\lambda}_\infty; \hbar\right) \\
+ 
\left( e^{\hbar \partial_{\lambda_0}} - 1 \right)
\left\{- \frac{1}{\hbar}\frac{\partial F_0}{\partial\lambda_0}
+ \frac{1}{2} \left( \nu_0  \frac{\partial^2 F_0}{\partial\lambda_0^2} 
+ \nu_1 \frac{\partial^2 F_0}{\partial\lambda_0 \lambda_1}
+ \nu_\infty \frac{\partial^2 F_0}{\partial\lambda_0 \lambda_\infty} 
\right) \right\}, 
\end{multline}
where $X_0$ is the operator given in \eqref{eq:Xj} for $j=0$.
Using \eqref{eq:Gauss_Xj}, we have
\begin{align} \label{eq:Gauss_X0_shift}
& 
e^{\hbar \partial_{\lambda_0} /2} 
X_0 F\left(\hat{\lambda}_0, \hat{\lambda}_1, \hat{\lambda}_\infty; \hbar\right) \\
& =
e^{(1 - \nu_0 - \nu_1 - \nu_\infty) \hbar \partial_{\lambda_0} /2} 
\log{(\lambda_0 + \lambda_1 + \lambda_{\infty})} 
+ e^{(1 - \nu_0 + \nu_1 - \nu_\infty) \hbar \partial_{\lambda_0} /2} 
\log{(\lambda_0 - \lambda_1 + \lambda_{\infty})} \notag \\
& \quad 
+ e^{(1 - \nu_0 - \nu_1 + \nu_\infty) \hbar \partial_{\lambda_0} /2} 
\log{(\lambda_0 + \lambda_1 - \lambda_{\infty})} 
+ e^{(1 - \nu_0 + \nu_1 + \nu_\infty) \hbar \partial_{\lambda_0} /2} 
\log{(\lambda_0 - \lambda_1 - \lambda_{\infty})} \notag \\
& \quad 
- e^{(1- \nu_0) \hbar \partial_{\lambda_0}/2}
\left\{ \log{(2 \lambda_0 + \hbar)} 
+ 2 \log{(2 \lambda_0)} + \log{(2 \lambda_0 - \hbar)} \right\}. \notag
\end{align}
To obtain the expression of the Voros coefficient, 
let us compute the action of the inverse operator of 
$\left( e^{\hbar \partial_{\lambda_0}} - 1 \right)$
to each terms appearing in the right hand-side of \eqref{eq:Gauss_X0_shift}.
For the purpose, Proposition \ref{prop:inverse-formula-2} is quite useful. 
For example, applying Proposition \ref{prop:inverse-formula-2} (i)
for $t=\nu_0 + \nu_1 + \nu_\infty$, $\lambda=\lambda_0$, 
$\mu = \lambda_1 + \lambda_{\infty}$, 
we can verify that the formal series
\begin{align}  
\label{eq:Gauss_Voros-calculation-1}
G(\underline{\lambda}, \underline{\nu}; \hbar)= 
& ~ 
\hbar^{-1} \Bigl( 
(\lambda_0 + \lambda_1 + \lambda_{\infty}) \log(\lambda_0 + \lambda_1 + \lambda_{\infty})
- \lambda_0 \Bigr)  - \frac{\nu_0 + \nu_1 + \nu_\infty}{2} 
\log{(\lambda_0 + \lambda_1 + \lambda_{\infty})}  \\
&
+ \sum_{m = 1}^{\infty} \frac{\hbar^m}{m(m+1)}  
\frac{B_{m+1}((\nu_0 + \nu_1 + \nu_\infty + 1)/2)}
{(\lambda_0 + \lambda_1 + \lambda_{\infty})^m} \notag 
\end{align}
is a particular solution of the difference equation 
\begin{equation}
\left( e^{\hbar \partial_{\lambda_0}} - 1 \right) 
G(\underline{\lambda}, \underline{\nu}; \hbar)
= 
e^{(1 - \nu_0 - \nu_1 - \nu_\infty) \hbar \partial_{\lambda_0} /2}  
\log{(\lambda_0 + \lambda_1 + \lambda_{\infty})}. 
\end{equation}
Namely, \eqref{eq:Gauss_Voros-calculation-1} coincides with 
the result of action of $\left( e^{\hbar \partial_{\lambda_0}} - 1 \right)^{-1}$ 
on the first term in the right hand side of \eqref{eq:Gauss_X0_shift} 
modulo constant terms.
Thanks to Proposition \ref{prop:inverse-formula-2},
we can compute the action of $\left( e^{\hbar \partial_{\lambda_0}} - 1 \right)^{-1}$ 
on the other terms in \eqref{eq:Gauss_X0_shift} similarly. 
Although \eqref{eq:Gauss_Voros-calculation-1} contains terms 
which are proportional to $\hbar^{-1}$ and $\hbar^0$,  
those terms cancel with the other terms 
\begin{align*}  
\frac{\partial F_0}{\partial {\lambda_0}} 
&= (\lambda_0 + \lambda_1 + \lambda_{\infty})\log{(\lambda_0 + \lambda_1 + \lambda_{\infty})} 
+ (\lambda_0 - \lambda_1 + \lambda_{\infty})\log{(\lambda_0 - \lambda_1 + \lambda_{\infty})} \\
& \quad
+ (\lambda_0 + \lambda_1 - \lambda_{\infty})\log{(\lambda_0 + \lambda_1 - \lambda_{\infty})}
+ (\lambda_0 - \lambda_1 - \lambda_{\infty})\log{(\lambda_0 - \lambda_1 - \lambda_{\infty})} \notag \\
& \quad 
- 4 \lambda_0 \log(2\lambda_0), \notag \\
\frac{\partial^2 F_0}{\partial {\lambda_0}^2} 
&= \log{(\lambda_0 + \lambda_1 + \lambda_{\infty})} 
+ \log{(\lambda_0 - \lambda_1 + \lambda_{\infty})}
+ \log{(\lambda_0 + \lambda_1 - \lambda_{\infty})}
+ \log{(\lambda_0 - \lambda_1 - \lambda_{\infty})} \\
&\quad 
- 4 \log{(2 \lambda_0)}, \notag \\
\frac{\partial^2 F_0}{\partial \lambda_0 \partial \lambda_1} 
&= \log{(\lambda_0 + \lambda_1 + \lambda_{\infty})} 
- \log{(\lambda_0 - \lambda_1 + \lambda_{\infty})}
+ \log{(\lambda_0 + \lambda_1 - \lambda_{\infty})}
- \log{(\lambda_0 - \lambda_1 - \lambda_{\infty})}, \\
\frac{\partial^2 F_0}{\partial \lambda_0 \partial \lambda_\infty} 
&= \log{(\lambda_0 + \lambda_1 + \lambda_{\infty})} 
+ \log{(\lambda_0 - \lambda_1 + \lambda_{\infty})}
- \log{(\lambda_0 + \lambda_1 - \lambda_{\infty})}
- \log{(\lambda_0 - \lambda_1 - \lambda_{\infty})}
\end{align*}
in the right hand side of \eqref{eq:Gauss-Voros-X0}.
In the end, we can verify that the formal series
\begin{align} \label{eq:explicit-Voros-Gauss-0}
& V^{(0)}(\underline{\lambda}, \underline{\nu}; \hbar)  = 
\sum_{m = 1}^{\infty} 
\frac{\hbar^m}{m (m + 1)}
\left\{
\frac{B_{m + 1} ((\nu_0 + \nu_1 + \nu_{\infty} + 1)/2)}
{(\lambda_0 + \lambda_1 + \lambda_{\infty})^m}
+ \frac{B_{m + 1} ((\nu_0 - \nu_1 + \nu_{\infty} + 1)/2)}
{(\lambda_0 - \lambda_1 + \lambda_{\infty})^m}
\right. \\ 
&
+ \frac{B_{m + 1} ((\nu_0 + \nu_1 - \nu_{\infty} + 1)/2)}
{(\lambda_0 + \lambda_1 - \lambda_{\infty})^m}
+ \frac{B_{m + 1} ((\nu_0 - \nu_1 - \nu_{\infty} + 1)/2)}
{(\lambda_0 - \lambda_1 - \lambda_{\infty})^m} 
\left.
- \frac{B_{m + 1}(\nu_0) + B_{m + 1}(\nu_0 + 1)}{(2\lambda_0)^m} 
\right\} \notag
\end{align}
satisfies the difference equation \eqref{eq:Gauss-Voros-X0}. 
Since we derive the formula \eqref{eq:explicit-Voros-Gauss-0} by solving 
the difference equation, the above expression only holds modulo 
formal series of $\hbar$ whose coefficients are 
independent of $\lambda_j$ a priori (cf. Proposition \ref{prop:almost-uniqueness-1}).
However, by a similar argument used in the previous subsection, 
we can conclude that \eqref{eq:explicit-Voros-Gauss-0} is indeed 
an explicit expression of the Voros coefficient for the quantum Gauss curve.
This completes the proof of Theorem \ref{thm:main-theorem-in-part2} (iv).

\section{Conclusion}
In this article we have established a relation between 
the free energy for spectral curves arising from 
the confluent family of the Gauss hypergeometric equations 
and the Voros coefficients of the associated quantum curves.
An obvious question is the possibility of generalization 
of our computational scheme to other examples; 
e.g., higher order or higher dimensional hypergeometric equations. 
We hope to discuss the problem in a future work. 

A generalization of Theorem \ref{thm:main-theorem-in-part2}
to higher genus spectral curves are also important.  
The Voros coefficients for Schr{\"o}dinger-type equations 
with higher genus classical limit are very interesting objects; 
they realize the cluster variables 
(or the Fock-Goncharov coordinates of \cite{GMN09})
when we discuss the parametric Stokes phenomenon. 
Several wall-crossing formulas are obtained in the context
(see \cite{GMN09, GMN12, FIMS} etc.). 
However, as is mentioned in \cite{BCD} and \cite{IS}, 
quantum curves for higher genus (or non-admissible)
spectral curves may include infinitely many $\hbar$-correction terms.
It is also noticed that the admissible spectral curves 
in the sense of \cite{BE} must be of genus $0$.
Therefore, we cannot expect that there is a straightforward generalization 
of our results for those spectral curves. 
Anyway, we need to understand how the theory of quantum curves 
should be generalized to higher genus spectral curves.


\appendix


\section{Contiguity relations of the Voros coefficient}
\label{sec:contiguity-relation}

Here we give a proof of Lemma \ref{lem:Gauss_Voros-parameter}
by using the contiguity relations of the Gauss hypergeometric equation.
A recent result of Oshima (see \cite{Oshima} and references therein)
shows that his framework is efficient in a study of linear differential
equations of rational coefficients theoretically and in practice.
As is discussed in \cite{IK},
it would give a computation method of Voros coefficients
in a unified and algorithmic manner.
Although it is enough to use the traditional contiguity relations
because we only discuss second order hypergeometric equations
in this paper, here we give them along Oshima's theory for future works.

\subsection{Contiguity relations of the Gauss hypergeometric equation}

In our study we need contiguity relations of the following operator 
defined though the middle convolution.
\begin{align}
\label{op:gauss}
P_{\text{Gauss}} (\kappa_1, \kappa_2, \mu)
:&= {\rm{RAd}}(\partial^{-\mu/\hbar})
{\rm{RAd}}(x^{\kappa_1/\hbar} (1-x)^{\kappa_2/\hbar}) \partial_x\\
&= x(1-x) (\hbar \partial_x)^2 + (c - (a + b + \hbar)x) 
(\hbar \partial_x) - ab.\notag\\[1ex]
\notag 
&
\left(
\begin{array}{c}
a = - (\mu + \kappa_1 + \kappa_2), \;
b =- \mu + \hbar, \;
c = - \mu - \kappa_1 + \hbar
  \\
\quad\Longleftrightarrow\quad
\kappa_1 = b - c, \;
\kappa_2 = -a + c - \hbar, \;
\mu = -b + \hbar
\end{array}\right)
\end{align}
Here ${\rm{RAd}}$ denotes the reduced adjoint operator (\cite[\S2.2]{Oshima}).
We specify the parameters 
\begin{equation}
a = \hat{\lambda}_0+\hat{\lambda}_1+\hat{\lambda}_\infty + \frac{\hbar}{2}, \quad
b = \hat{\lambda}_0+\hat{\lambda}_1-\hat{\lambda}_\infty + \frac{\hbar}{2}, \quad
c = 2\hat{\lambda}_0 + \hbar
\end{equation}
so that the equation $P_{\rm Gauss} \psi_{\rm Gauss} = 0$ 
is equivalent to the quantum Gauss curve \eqref{eq:Gauss_eq(d/dx)}
by a gauge transform (cf. \eqref{eq:gauge-tramsform-Gauss}).
These relations are equivalent to 
\begin{equation}
\hat{\lambda}_0 = - \frac{\mu}{2} - \frac{\kappa_1}{2}, \quad 
\hat{\lambda}_1 = - \frac{\mu}{2} - \frac{\kappa_2}{2}, \quad
\hat{\lambda}_\infty = - \frac{\kappa_1}{2} - \frac{\kappa_2}{2} - \frac{\hbar}{2}
\end{equation}
in terms of the parameters related to the middle convolution.

We will use the following contiguity relations of 
the operator $P_{\rm Gauss}$ (cf. \cite[\S5 and \S4]{Oshima}).
\begin{thm}
\label{thm:contiguity}
The differential operator \eqref{op:gauss} satisfies 
the following contiguity relations.
\begin{align*}
\partial_x \circ P_{\text{\rm Gauss}} (\kappa_1, \kappa_2, \mu)
&=
P_{\text{\rm Gauss}} (\kappa_1, \kappa_2, \mu - \hbar) \circ \partial_x,
\\
\big(x + (\hbar - \mu) (\hbar\partial_x)^{-1} \big)
\circ P_{\text{\rm Gauss}} (\kappa_1, \kappa_2, \mu)
&=
P_{\text{\rm Gauss}} (\kappa_1 + \hbar, \kappa_2, \mu) \circ
\big(x -\mu (\hbar \partial_x)^{-1}\big),
\\
\big(x-1 + (\hbar - \mu)(\hbar \partial_x)^{-1}\big)
\circ P_{\text{\rm Gauss}} (\kappa_1, \kappa_2, \mu)
&=
P_{\text{\rm Gauss}} (\kappa_1, \kappa_2+\hbar, \mu) \circ 
\big(x -1 - \mu (\hbar \partial_x)^{-1}\big).
\end{align*}
\end{thm}

In particular, the parameter shift 
$(\kappa_1, \kappa_2, \mu) \mapsto (\kappa_1, \kappa_2, \mu - \hbar)$ 
in the first line can be translated into the shift 
$(\nu_0,\nu_1,\nu_\infty) \mapsto (\nu_0-1, \nu_1-1, \nu_\infty)$
of the parameters $\nu_j$'s.

\subsection{Proof of Lemma \ref{lem:Gauss_Voros-parameter}}
\label{subsection:proof-of-contiguity-Voros}

Now we prove Lemma \ref{lem:Gauss_Voros-parameter} 
with the aid of Theorem \ref{thm:contiguity}. 
The following proof is based on the same ideas of 
\cite{Takei08} and \cite{IK}.
We will only give a proof of \eqref{eq:Gauss-Voros-contiguity-1}
in Lemma \ref{lem:Gauss_Voros-parameter}
since the other equalities can be derived similarly 
(see Remark \ref{rem:sketch-of-proof-of-contiguity}).  

\subsubsection{Preliminary for the proof}
Let $\psi_{\rm Gauss}(x,\hbar)$ be the WKB solution to 
the equation 
\begin{equation} \label{eq:Gauss-equation-appendix}
P_{\rm Gauss} \psi_{\rm Gauss} = 0,
\end{equation} 
and 
$S_{\rm Gauss} = \sum_{m \ge -1} \hbar^{m} S_{{\rm Gauss}, m}(x)$ 
be its logarithmic derivative. 
In view of \eqref{eq:gauge-tramsform-Gauss}, 
the logarithmic derivative $S$ of the WKB solution 
of the quantum Gauss curve \eqref{eq:Gauss_eq(d/dx)} 
(which is considered in \S \ref{subsec:proof-gauss}) 
is related to $S_{\rm Gauss}$ as follows:
\begin{equation} \label{eq:relations-of-S-Gauss}
S_{\rm Gauss}(x,\hbar) = S(x,\hbar) 
- \left( \frac{\hat{\lambda}_0}{\hbar} 
+ \frac{\nu_{0+}+\nu_{0-}}{2} \right) \frac{1}{x} 
- \left( \frac{\hat{\lambda}_1}{\hbar} 
+ \frac{\nu_{1+}+\nu_{1-}}{2} \right) \frac{1}{x-1}.
\end{equation}
The relation \eqref{eq:relations-of-S-Gauss} implies 
that the formal series $S$ and $S_{\rm Gauss}$ 
differ only by first two terms 
(i.e., terms proportional to $\hbar^{-1}$ and $\hbar^{0}$). 
Therefore, the Voros coefficients 
\begin{equation}
V_{\rm Gauss}^{(j)}(\underline{\lambda}, \underline{\nu};\hbar) 
:=
\int_{\gamma_j} \left( 
S_{\rm Gauss}(x,\hbar) - \hbar^{-1} S_{{\rm Gauss}, -1}(x) 
-  S_{{\rm Gauss}, 0}(x) \right) \, dx
\end{equation}
of the equation \eqref{eq:Gauss-equation-appendix} 
defined for the same path $\gamma_j$ ($j \in \{0,1,\infty \}$)
given in \S \ref{subsec:main-theorem} 
coincides with those for the quantum Gauss curve; that is, 
$V_{\rm Gauss}^{(j)} = V^{(j)}$ holds for each $j \in \{0,1,\infty \}$.
Therefore, our task is to show that the Voros coefficient 
$V_{\rm Gauss}^{(0)}$ satisfies \eqref{eq:Gauss-Voros-contiguity-1}.

We will use the symbols $x_1$ and $x_2$ for 
points on $\gamma_0$ which eventually tend to the 
end point and initial point of $\gamma_0$.
We will also use the following obvious expression:
\begin{equation}
\left( V^{(0)}(\underline{\lambda}, \underline{\nu};\hbar) 
 = \right) ~
V^{(0)}_{\rm Gauss}(\underline{\lambda}, \underline{\nu};\hbar) 
 =
\lim_{\substack{x_1 \to 0 \\ x_2 \to 0^\dagger}} 
V_{\rm Gauss}^{x_1,x_2}(\underline{\lambda}, \underline{\nu};\hbar) 
\end{equation}
with
\begin{equation} \label{eq:pre-Voros-Gauss}
V_{\rm Gauss}^{x_1,x_2}(\underline{\lambda}, \underline{\nu};\hbar)  
:=  
\int^{x_1}_{x_2} \left( 
S_{\rm Gauss}(x,\hbar) - \hbar^{-1} S_{{\rm Gauss}, -1}(x) 
-  S_{{\rm Gauss}, 0}(x) \right) \, dx.
\end{equation}
Here the integration in \eqref{eq:pre-Voros-Gauss} 
is taken along a part of $\gamma_0$, 
and $0$ and $0^\dagger$ are the end point and initial point 
of $\gamma_0$, respectively. 
(These points are preimages of $x=0$ by the projection 
$\Sigma \to {\mathbb P}^1$, or in other words, 
$0$ corresponds to $\beta_{0+}$ on $z$-plane 
while $0^\dagger$ corresponds to $\beta_{0-}$.)

In the following proof we need to analyze the 
behavior of $S_{{\rm Gauss}, m}(x)$ when $x$ approaches to 
$0$ and $0^\dagger$.
A straightforward computation shows
\begin{eqnarray}
\label{eq:S-Gauss--1}
S_{{\rm Gauss}, -1}(x) & = & 
\begin{cases}
O(x^{0}) & \text{when $x \to 0$}, 
\\[+.5em]
-\dfrac{2\lambda_0}{x} + O(x^{0})  & \text{when $x \to 0^\dagger$},
\end{cases} \\[+.5em]
\label{eq:S-Gauss-0} 
S_{{\rm Gauss}, 0}(x) & = & 
\begin{cases}
O(x^{0}) & \text{when $x \to 0$}, \\[+.5em]
\dfrac{\nu_0}{x} + O(x^{0}) & \text{when $x \to 0^\dagger$},
\end{cases} \\[+.5em]
\label{eq:S-Gauss-m}
S_{{\rm Gauss}, m}(x) & = & O(x^0) \quad
\text{when $x \to 0$ and $x \to 0^\dagger$}.
\end{eqnarray}
The last property \eqref{eq:S-Gauss-m} is a consequence 
of the fact that the correlation function 
$W_{g,n}(z_1,\dots,z_n)$ is holomorphic at 
$z_i = \beta_{0\pm}$ (which are not ramification points). 
Thus $S_{\rm Gauss}$ behaves as 
\begin{equation} \label{eq:behavior-of-S-Gauss}
S_{\rm Gauss}(x,\hbar) = 
\begin{cases}
A(\hbar) + O(x^1) & \text{when $x \to 0$}, \\[+.7em]
- \dfrac{2\hat{\lambda}_0}{\hbar x} 
+ O(x^0) 
& \text{when $x \to 0^\dagger$}
\end{cases}
\end{equation}
near the origin. Here $A(\hbar)$, which is independent of $x$, 
can be determined by looking the behavior of coefficients 
of \eqref{eq:Gauss-equation-appendix} near the origin; 
the explicit form is given by 
\begin{equation} \label{eq:expression-of-A}
A(\hbar) = \frac{
(2\hat{\lambda}_0 + 2\hat{\lambda}_1 + 2\hat{\lambda}_\infty + \hbar)
(2\hat{\lambda}_0 + 2\hat{\lambda}_1 - 2\hat{\lambda}_\infty + \hbar)}
{4\hbar(2\hat{\lambda}_0+\hbar)}.
\end{equation}
(Precisely speaking, since $S_{\rm Gauss}(x,\hbar)$ is a 
formal series of $\hbar$, so does $A(\hbar)$. 
Therefore, $A(\hbar)$ should be understood as 
the Taylor expansion of the right hand-side of the formula 
\eqref{eq:expression-of-A}.)

\subsubsection{Derivation of contiguity relation 
for the Voros coefficients}
Now we use Theorem \ref{thm:contiguity} to derive the 
contiguity relations satisfied by the Voros coefficient.
It follows from the first relation in 
Theorem \ref{thm:contiguity} that
\begin{equation}
\partial_x \psi_{\rm Gauss}(x,\nu_0,\nu_1) = 
{\rm const} \times \psi_{\rm Gauss}(x,\nu_0-1,\nu_1-1)
\end{equation}
holds. 
(In this proof we omit the dependence on other variables since 
we are mainly interested in dependences on $\nu_0$ and $\nu_1$.)
Then, the logarithmic derivative $S_{\rm Gauss}$ 
of $\psi_{\rm Gauss}$ satisfies
\begin{equation} \label{eq:contiguity-WKB-Gauss}
S_{\rm Gauss}(x,\nu_0, \nu_1) - S_{\rm Gauss}(x,\nu_0-1, \nu_1-1) 
= - \frac{d}{dx} \log S_{\rm Gauss}(x,\nu_0, \nu_1). 
\end{equation}
Integrating the both sides from $x_2$ to $x_1$, we have
\begin{multline} \label{eq:contiguity-pre-Voros-Gauss}
V_{\rm Gauss}^{x_1,x_2}(\nu_0, \nu_1) 
- V_{\rm Gauss}^{x_1,x_2}(\nu_0-1, \nu_1-1) 
= - \log \left( 
\frac{S_{\rm Gauss}(x_1,\nu_0, \nu_1)}
{S_{\rm Gauss}(x_2,\nu_0, \nu_1)} \right) \\
- \sum_{m \in \{-1,0 \} } \hbar^{m} \left(
\int^{x_1}_{x_2}S_{{\rm Gauss}, m}(x,\nu_0, \nu_1) \, dx
- \int^{x_1}_{x_2}S_{{\rm Gauss}, m}(x,\nu_0-1, \nu_1-1) \, dx  
\right).
\end{multline}

In order to derive a difference equation satisfied by 
$V_{\rm Gauss}^{(0)}$, let us look at the behavior 
of the right hand-side of \eqref{eq:contiguity-pre-Voros-Gauss}
when $x_1 \to 0$ and $x_2 \to 0^\dagger$.
Since $S_{{\rm Gauss}, -1}(x)$ is independent of $\nu_i$'s, we have
$S_{{\rm Gauss}, -1}(x,\nu_0, \nu_1) 
- S_{{\rm Gauss}, -1}(x,\nu_0-1, \nu_1-1) = 0$.
Moreover, a straightforward computation shows 
\begin{align}
& \int^{x}S_{{\rm Gauss}, 0}(x,\nu_0, \nu_1) \, dx
- \int^{x}S_{{\rm Gauss}, 0}(x,\nu_0-1, \nu_1-1) \, dx \\
& \nonumber \quad = 
- \frac{1}{2} \log \left( 
\frac{2 \lambda_0 
\sqrt{\lambda_\infty^2 x^2 - 
(\lambda_0^2-\lambda_1^2+\lambda_\infty^2)x + \lambda_0^2} 
- (\lambda_0^2-\lambda_1^2+\lambda_\infty^2)x + 2\lambda_0^2}{x^2} 
\right) \\
& \nonumber \qquad 
+ \frac{1}{2} \log \left( 
- 2 \lambda_1 \sqrt{\lambda_\infty^2 x^2 - 
(\lambda_0^2-\lambda_1^2+\lambda_\infty^2)x + \lambda_0^2} 
+ (\lambda_0^2-\lambda_1^2-\lambda_\infty^2)x - 
(\lambda_0^2+\lambda_1^2-\lambda_\infty^2)
\right) \\
\nonumber 
& = 
\begin{cases}
\displaystyle 
\frac{1}{2} \log \left( 
- \frac{4 \lambda_0^2}{(\lambda_0+\lambda_1-\lambda_\infty)
(\lambda_0+\lambda_1+\lambda_\infty)}  
\right) + O(x^1)  & \text{when $x \to 0$}, \\[+1.em]
\displaystyle 
\frac{1}{2} \log \left( 
- \frac{(\lambda_0+\lambda_1-\lambda_\infty)
(\lambda_0+\lambda_1+\lambda_\infty)}{4 \lambda_0^2} x^2  
\right) + O(x^1)  & \text{when $x \to 0^\dagger$}.
\end{cases}
\end{align}
Therefore, the right hand-side of \eqref{eq:contiguity-pre-Voros-Gauss} 
coincides with 
\begin{multline}
- \log A(\hbar) - \frac{1}{2} \log \left( 
- \frac{4 \lambda_0^2}{(\lambda_0+\lambda_1-\lambda_\infty)
(\lambda_0+\lambda_1+\lambda_\infty)}  
\right) \\ 
+ \log \left( - \frac{2 \hat{\lambda}_0}{\hbar x_2} \right)
+ \frac{1}{2} \log \left( 
- \frac{(\lambda_0+\lambda_1+\lambda_\infty)
(\lambda_0+\lambda_1-\lambda_\infty)}{4 \lambda_0^2} x_2^2  
\right) \\
= 
- \log\left[ 
\frac{\lambda_0^2}{\hat{\lambda}_0(\hat{\lambda}_0 + {\hbar}/{2})}
\,
\frac{(\hat{\lambda}_0+\hat{\lambda}_1
+\hat{\lambda}_\infty + {\hbar}/{2})}
{(\lambda_0+\lambda_1+\lambda_\infty)} \,
\frac{(\hat{\lambda}_0+\hat{\lambda}_1
- \hat{\lambda}_\infty + {\hbar}/{2})}
{(\lambda_0+\lambda_1-\lambda_\infty)}
\right]
\end{multline}
modulo terms which vanish 
in the limit $x_1 \to 0$ and $x_2 \to 0^\dagger$. 
Therefore, taking this limit in the both sides of 
\eqref{eq:contiguity-pre-Voros-Gauss}, 
we obtain the desired relation \eqref{eq:Gauss-Voros-contiguity-1}. 
This completes the proof.  

\begin{rem}\label{rem:sketch-of-proof-of-contiguity}
Combining the equalities in Theorem \ref{thm:contiguity}, 
we can derive 
\begin{eqnarray}
\left( x \hbar \partial_x - \mu + \hbar \right)
\psi_{\rm Gauss}(\nu_0,\nu_1, \nu_\infty) & = &  
{\rm cosnt} \times 
\psi_{\rm Gauss}(\nu_0,\nu_1-1, \nu_\infty+1), \\
\left( (x-1) \hbar \partial_x - \mu + \hbar \right)
\psi_{\rm Gauss}(\nu_0,\nu_1, \nu_\infty) & = &  
{\rm cosnt} \times 
\psi_{\rm Gauss}(\nu_0-1,\nu_1, \nu_\infty+1)
\end{eqnarray}
as counterparts of \eqref{eq:contiguity-WKB-Gauss}. 
Then we can perform a similar calculation as above 
to derive the other equalities 
\eqref{eq:Gauss-Voros-contiguity-2} 
and \eqref{eq:Gauss-Voros-contiguity-3} 
in Lemma \ref{lem:Gauss_Voros-parameter}.  
We omit the details.
\end{rem}


\section{Bernoulli numbers and difference equations}
\label{sec:bernoulli}

Here we summarize several useful formulas of Bernoulli numbers 
which are applied to solve difference equations satisfied by 
free energies and Voros symbols in \S \ref{sec:Voros-vs-TR}.

\subsection{Definitions of Bernoulli numbers and Bernoulli polynomials}

The Bernoulli number $\{B_n\}_{n \geq 0}$ can be defined through the generating function as
\begin{equation}
\label{def:Bernoulli}
\frac{w}{e^w - 1} = \sum_{n = 0}^{\infty} B_n \frac{w^n}{n!}.
\end{equation}
From this definition we find
\begin{equation} \label{def:Bernoulli-2}
B_0 = 1, \quad B_1 = -\frac{1}{2}, \quad B_2 = \frac{1}{6},
\quad B_4 = -\frac{1}{30},
\quad B_{2k + 1} = 0\;\text{for}\; k \geq 1.
\end{equation}
The Bernoulli polynomials, which we denote by $\{B_n(t)\}_{n \geq 0}$,
can also be defined through the generating function:
\begin{equation}
\label{def:BernoulliPoly}
\frac{w e^{t w}}{e^w - 1} = \sum_{ n = 0}^{\infty} B_n(t) \frac{w^n}{n!}.
\end{equation}
Other useful expressions are also known, e.g.,
\begin{equation}
B_n(t) = \sum_{k = 0}^n \binom{n}{k} B_{n -k} t^k.
\end{equation}
Important relations for the Bernoulli polynomials are
\begin{align}
\label{bernoulli:reflection}
 B_n (1-t) &= (-1)^n B_n(t), \\
\label{bernoulli:minus}
(-1)^n B_n(-t) &= B_n(t) + n t^{n-1}, \\
\label{bernoulli:multiplication}
  B_n (mx) &= m^{n-1} \sum_{k=0}^{m-1} B_n \left( x + \frac{k}{m} \right)
  \quad (m \ge 1).
\end{align}
We will also use formulas about special values of the Bernoulli polynomials:
\begin{align}
\label{bernoulli:value_at_zero}
 B_n (0)&= B_n, \\
\label{bernoulli:value_at_a_half}
B_n\left(\frac{1}{2}\right) &= \left(\frac{1}{2^{n-1}} - 1\right) B_n.
\end{align}

We use these Bernoulli numbers and Bernoulli polynomials to give
particular solutions of some difference equations.
Its basic idea is due to \cite{Aoki-Tanda}.

\begin{prop}
\label{prop:Bernoulli-generating}
\begin{alignat*}{2}
\text{{\rm{(i)}}}
&\quad
& \frac{1}{e^w - 1} &= \frac{1}{w}
+ \sum_{n = 0}^{\infty} \frac{B_{n + 1}}{\, n + 1 \,} \frac{\, w^n \,}{\, n! \,}.
\\
\text{{\rm{(ii)}}}
&\quad
& \frac{e^w}{(e^w - 1)^2} 
&= 
\frac{1}{w^2}
- \sum_{n = 0}^{\infty} \frac{B_{n + 2}}{\, n + 2 \,} \frac{\, w^n \,}{\, n! \,}.
\end{alignat*}
\end{prop}

\begin{proof}
The claim (i) follows from \eqref{def:Bernoulli}.
By differentiating (i) with respect to $w$, we get (ii).
\end{proof}

\begin{prop}
\label{prop:BernoulliPoly-generating}
\begin{alignat*}{2}
\text{{\rm{(i)}}}
&\quad
& \frac{e^{t w}}{e^w - 1} &=
\frac{1}{w}
+ \sum_{n = 0}^{\infty} \frac{B_{n + 1}(t)}{\, n + 1 \,} \frac{\, w^n \,}{\, n! \,}.
\\
\text{{\rm{(ii)}}}
&\quad
& \frac{e^{(1 + t)w}}{(e^w - 1)^2} &=
\frac{1}{w^2} + \frac{t}{w} +
\sum_{n = 0}^{\infty}
\left\{t \frac{\, B_{n + 1}(t) \,}{\, n + 1 \,} 
- \frac{\, B_{n + 2}(t) \,}{\, n + 2 \,}\right\} \frac{\, w^n \,}{\, n! \,}.
\end{alignat*}
\end{prop}

\begin{proof}
The first statement is a direct consequence of \eqref{def:BernoulliPoly}.
After differentiating (i) with respect to $w$, we have
\begin{equation}
\frac{t e^{tw}}{e^w - 1} - \frac{e^{(t+1)w}}{(e^w-1)^2}
= - \frac{1}{w^2}
+ \sum_{n = 0}^{\infty}
\frac{\, B_{n + 2}(t)\, }{n +2}
\frac{\, w^n \,}{\, n! \,}.
\end{equation}
Then we substitute (i) into the first term of the right-hand side to
obtain (ii).
\end{proof}

\subsection{Application to difference equations}
\label{subsection:solving-difference-eq}

Let $X$ be the difference operator defined by 
\begin{equation}
X = e^{- \hbar \partial_{\lambda}} 
\left( e^{\hbar \partial_{\lambda}} - 1 \right)^2 ~~:~~
F(\lambda) \mapsto  
F(\lambda+\hbar) - 2F(\lambda) + F(\lambda-\hbar). 
\end{equation}
In \S \ref{subsection:solving-difference-F} we need to solve 
difference equations of the form $X F(\lambda) = R(\lambda)$
to obtain the explicit formula for free energies. 
For the purpose, we use
\begin{prop}
\label{prop:difference-eq:sol}
Particular solutions of
a difference equation
\begin{equation} \label{prop:difference-eq:sol:tmp:1}
X F(\lambda) = 
F(\lambda + \hbar) - 2 F(\lambda) + F(\lambda - \hbar) = 
R(\lambda)
\end{equation}
for
\begin{itemize}
\item[{\rm{(i)}}] 
$R(\lambda) = \log (\lambda+\mu)$ 
(where $\mu$ is independent of $\lambda$), and
\item[{\rm{(ii)}}] 
$R(\lambda) = 2\log(2\lambda) +  \log (2\lambda + \hbar) +  \log (2\lambda - \hbar)$,
\end{itemize}
are respectively given by
\begin{alignat*}{2}
\text{{\rm{(i)}}}&\quad&
F(\lambda)
&= 
(\hbar \partial_{\lambda})^{-2} \log(\lambda+\mu)
- \frac{1}{12} \log (\lambda + \mu)
+ \sum_{n = 2}^{\infty} \frac{B_{2n}}{2n (2n - 2)}
\left(\frac{\hbar}{\lambda+\mu}\right)^{2n-2}.
\\
\text{{\rm{(ii)}}}&\quad&
F(\lambda)
&=
4(\hbar \partial_{\lambda})^{-2} \log (2\lambda)
- \frac{1}{12} \log \lambda
+ \sum_{n = 2}^{\infty}
\frac{B_{2n}}{2 n (2n - 2)} 
\left(\frac{\hbar}{2\lambda}\right)^{2n - 2}.
\end{alignat*}
Here $\partial_\lambda^{-1} f(\lambda)$ means 
an indefinite integral of $f(\lambda)$.
\end{prop}

\begin{proof}
\begin{itemize}
\item[(i)]
It follows from
\begin{equation}
e^{-w} (e^w - 1)^2
\left\{
\frac{1}{w^2}
- \sum_{n = 0}^{\infty} \frac{B_{n + 2}}{\, n + 2 \,} \frac{\, w^n \,}{\, n! \,}
\right\}
= 1
\end{equation}
(cf. Proposition \ref{prop:Bernoulli-generating} (ii)) that
\begin{equation}
e^{-\hbar\partial_{\lambda}} (e^{\hbar\partial_{\lambda}} - 1)^2
\left\{
(\hbar\partial_{\lambda})^{-2}
- \sum_{n = 0}^{\infty} \frac{B_{n + 2}}{\, n + 2 \,} 
\frac{\, (\hbar\partial_{\lambda})^n \,}{\, n! \,}
\right\}
= {\rm{id}}.
\end{equation}
Hence we find that
\begin{equation}
\left\{
(\hbar\partial_{\lambda})^{-2}
- \sum_{n = 0}^{\infty} \frac{B_{n + 2}}{\, n + 2 \,} 
\frac{\, (\hbar\partial_{\lambda})^n \,}{\, n! \,}
\right\}
\log (\lambda+\mu)
\end{equation}
gives a solution of \eqref{prop:difference-eq:sol:tmp:1}
for $R(\lambda) = \log(\lambda+\mu)$.
Thus we have proved (i). 
(Note that the Bernoulli number $B_{n}$ vanishes 
for odd $n \ge 3$ as we see in \eqref{def:Bernoulli-2}.)

\item[(ii)]
Let us first solve auxiliary difference equations 
\begin{eqnarray}
X G(\lambda) & = & \log(2\lambda), \label{eq:auxiliary-1} \\ 
X H_{\pm}(\lambda) & = & \log(2\lambda \pm \hbar). \label{eq:auxiliary-2}
\end{eqnarray}
By a similar method used in the proof of (i), 
we can verify that 
\begin{equation}
G(\lambda)=(\hbar \partial_\lambda)^{-2} \log(2\lambda) 
- \frac{1}{12} \log(2\lambda)  
+ \sum_{n=2}^{\infty} \frac{B_{2n}}{2n(2n-2)} \left( \frac{\hbar}{\lambda} \right)^{2n-2}
\end{equation}
satisfies \eqref{eq:auxiliary-1}. 
On the other hand, \eqref{eq:auxiliary-2} is equivalent to 
\begin{equation}
e^{- (1 \pm \frac{1}{2}) \hbar \partial_{\lambda}}
(e^{\hbar \partial_{\lambda}} - 1)^2 H_{\pm}(\lambda) = \log (2\lambda).
\end{equation}
Then, Proposition \ref{prop:BernoulliPoly-generating} (ii) with $t = \pm 1/2$ 
implies that the formal series
\begin{multline}
H_{\pm}(\lambda)
=
(\hbar \partial_{\lambda})^{-2} \log (2\lambda)
\pm \frac{1}{2} 
(\hbar \partial_{\lambda})^{-1} \log (2\lambda)
+ \frac{1}{24} \log (2\lambda)
\\
- \sum_{n = 1}^{\infty}
\left\{\pm \frac{B_{n + 1}(\pm 1/2)}{2 n (n + 1)} 
- \frac{B_{n + 2}(\pm 1/2)}{n (n + 2)}\right\}
\left(- \frac{\hbar}{\lambda}\right)^n
\end{multline}
is a particular solution of \eqref{eq:auxiliary-2}. 
Therefore, the formal series 
$F(\lambda) = 2G(\lambda) + H_{+}(\lambda) + H_{-}(\lambda)$
gives a particular solution of the original problem 
$XF(\lambda) = 2\log(2\lambda)+\log(2\lambda+\hbar)+\log(2\lambda-\hbar)$. 
Then, the claim (ii) follows from 
the identities \eqref{bernoulli:minus} and 
\eqref{bernoulli:value_at_a_half} of Bernoulli polynomials.
\end{itemize}
\end{proof}

The following result is useful to obtain an explicit expression of Voros coefficients.  

\begin{prop} \label{prop:inverse-formula-2}
Particular solutions of a difference equation
\begin{equation}
\left( e^{\hbar \partial_{\lambda}} - 1 \right) F(\lambda) 
= e^{(1 - t) \hbar \partial_{\lambda} /2}  R(\lambda)
\end{equation}
for 
\begin{itemize}
\item[{\rm (i)}] $R = \log(\lambda+\mu)$ 
(where $\mu$ is independent of $\lambda$), and 
\item[{\rm (ii)}] $R = 2 \log{(2 \lambda)} 
+ \log{(2 \lambda + \hbar)}  + \log{(2 \lambda - \hbar)}$
\end{itemize} 
are respectively given by 
\begin{alignat*}{2}
\text{{\rm{(i)}}}&\quad&
F(\lambda)
&= 
(\hbar \partial_{\lambda})^{-1} \log{(\lambda + \mu)} 
- \frac{t}{2} \log{(\lambda + \mu)}  
+ \sum_{m = 1}^{\infty} \frac{\hbar^m}{m(m+1)} 
\frac{B_{m+1}((t + 1)/2)}{(\lambda + \mu)^m}, \\
\text{{\rm{(ii)}}}&\quad&
F(\lambda) &=
4(\hbar \partial_{\lambda})^{-1} \log{(2 \lambda)} 
- 2 t \log{(2 \lambda)} 
+ \sum_{m = 1}^{\infty} \frac{\hbar^m}{m(m+1)}  
\frac{B_{m+1}(t) + B_{m+1}(t + 1)}{(2 \lambda)^m}.
\end{alignat*}
\end{prop}
\begin{proof}
The claim (i) follows 
directly from Proposition \ref{prop:BernoulliPoly-generating} (i) 
and \eqref{bernoulli:reflection}. 
The second claim (ii) can be verified by a similar argument used 
in the proof of Proposition \ref{prop:difference-eq:sol} (ii).
\end{proof}

In the last, we study solutions of homogeneous difference equations
to discuss a uniqueness of solutions. 

\begin{prop}
\label{prop:almost-uniqueness-1}
A formal power series solution
$F(\lambda) = \sum_{n \geq 0} F_n(\lambda) h^{n}$
of a difference equation
\begin{equation}
F(\lambda + \hbar) - F(\lambda) = 0
\end{equation}
should satisfy $\partial_{\lambda}F_n = 0$ ($n \geq 0$).
Here we regard $\hbar$ as a small parameter and
$F_n(\lambda + \hbar)$ for $n \geq 0$ is defined by the Taylor expansion:
\begin{equation}
\label{def:difference-op}
F_n(\lambda + \hbar)
= \sum_{k = 0}^{\infty} \frac{\hbar^k}{k!}\frac{d^k}{d\lambda^k} F_n(\lambda).
\end{equation}
\end{prop}

\begin{proof}
From \eqref{def:difference-op}, we have
\begin{equation}
F(\lambda+ \hbar) - F(\lambda)
= \sum_{n = 0}^{\infty}
\sum_{k = 0}^{n-1} \frac{\hbar^n}{(n-k)!} \frac{d^{n-k}}{d\lambda^{n-k}} F_k.
\end{equation}
Hence $F(\lambda+\hbar) - F(\lambda) = 0$ implies
\[
\frac{d}{d\lambda} F_0 = 0, \quad
\frac{d}{d\lambda} F_1 + \frac{1}{2!} \frac{d^2}{d\lambda^2} F_0 = 0, \quad
\frac{d}{d\lambda} F_2
+ \frac{1}{2!} \frac{d^2}{d\lambda^2} F_1
+ \frac{1}{3!} \frac{d^3}{d\lambda^3} F_0 = 0, 
\dots
\]
and so on. 
Hence, by induction, we can show $\partial_{\lambda} F_n = 0$
for $n \geq 0$.
\end{proof}

\begin{prop} \label{prop:almost-uniqueness-2}
A formal power series solution
$F(\lambda) = \sum_{n \geq 0} F_n(\lambda) h^{n}$
of a difference equation
\begin{equation}
F(\lambda + \hbar) - 2 F(\lambda) + F(\lambda - \hbar)= 0
\end{equation}
should satisfy $\partial_{\lambda}^2F_n = 0$ $(n \geq 0)$.
\end{prop}
This can be proved in the same way as 
Proposition \ref{prop:almost-uniqueness-1}.
We omit the details.


\end{document}